\RequirePackage{fix-cm}

\documentclass{imsart}

\usepackage{booktabs}
\usepackage{xcolor}
\usepackage{graphicx}
\usepackage{caption}
\usepackage{subcaption}
\usepackage{latexsym}
\usepackage{amssymb}

\usepackage{changepage}
\newcommand{\orb}[1]{\normalsize{\ensuremath{\mathsf{#1}}}\normalsize}
\newcommand{\orbcap}[1]{\ensuremath{\mathsf{#1}}}
\newcommand{\xcap}{\times}
\newcommand{\handle}{\circ}
\renewcommand{\star}{\ast}

\usepackage{amsmath,amsthm}

\DeclareMathOperator{\cl}{cl}
\DeclareMathOperator{\inter}{int}

\theoremstyle{plain}
\newtheorem*{theorem*}{Theorem}
\newtheorem*{proposition*}{Proposition}
\newtheorem{theorem}{Theorem}[section]
\newtheorem{lemma}[theorem]{Lemma}

\theoremstyle{definition}
\newtheorem{definition}[theorem]{Definition}

\begin{document}
\begin{frontmatter}

\title{Tiling the Euclidean and Hyperbolic planes with ribbons }

 \runtitle{Ribbon tilings}        	

\begin{aug}
 \author{\fnms{Benedict} \snm{Kolbe}\ead[label=e1]{kolbe@math.tu-berlin.de}}
 \and
  \author{\fnms{Vanessa}  \snm{Robins}\corref{}\thanksref{t2}\ead[label=e2]{vanessa.robins@anu.edu.au}},

  \thankstext{t2}{Corresponding author supported by ARC Future Fellowship FT140100604.}

\runauthor{Kolbe and Robins}

  \affiliation{}

  \address{Institute of Mathematics, Technische Universit\"at Berlin \\  \printead{e1}}
  \address{Research School of Physics, The Australian National University, Canberra \\ \printead{e2}}
  
\end{aug}

\begin{abstract}
We describe a method to classify crystallographic tilings of the Euclidean and hyperbolic planes by tiles whose stabiliser group contains translation isometries or whose topology is not that of a closed disk. We tackle this problem from two different viewpoints, one with constructive techniques to enumerate such tilings and the other from a viewpoint of classification. 
The methods are purely topological and generalise Delaney-Dress combinatorial tiling theory.  The classification is up to equivariant equivalence and is achieved by viewing tilings as decorations of orbifolds. 
\end{abstract}


\end{frontmatter}

\section{Introduction}
\label{intro}

Patterns built from repeating motifs appear in all cultures and have long been studied in art, mathematics, engineering and science.  
Most mathematical work has focussed on patterns in the Euclidean plane (the book ``Tilings and Patterns''  by Gr\"unbaum and Shephard~\cite{grunbaum_tilings_1987} contains a comprehensive survey of the field up to the mid 1980s) but the importance of hyperbolic geometry as a model for natural forms is increasingly recognised~\cite{language_of_shape,pedersen_polyhedra_2018,KAPFER20116875,Sharon:2002df}.
An example that inspires the work in this paper is the discovery that star co-polymer systems consisting of three mutually immiscible arms can self-assemble into structures modelled by stripes on the gyroid triply periodic minimal surface~\cite{kirkensgaard_hierarchical_2014,deCampo2017}.  
The gyroid surface has genus three in its smallest side-preserving translational unit cell, and therefore has the hyperbolic plane as its simply-connected covering space. 
Its 3d space-group symmetries induce a non-euclidean crystallographic group generated by hyperbolic isometries that are known explicitly~\cite{Sadoc1989,Robins2005}.
Stripe patterns on the gyroid lift via the covering map to tilings of the hyperbolic plane by infinitely long strips, or ribbons. 
The defining property of a ribbon tile is the existence of a translation isometry that maps a given tile back onto itself. 
This paper considers the general question of how to describe and enumerate crystallographic tilings of the Euclidean and hyperbolic planes by ribbon tiles.  

The Euclidean case of striped patterns is described in Section 6.5 of \cite{grunbaum_tilings_1987} citing earlier work by Wollny~\cite{Wollny1981}; there are 26 distinct types of crystallographic ribbon tilings of the Euclidean plane. 
To obtain corresponding results for the hyperbolic plane we require different mathematical techniques to those used by Gr\"unbaum and Shephard.  
Our approach derives from the perspective of Dress \emph{et al.} who developed the field of combinatorial tiling theory~\cite{DRESS1987,Dresshu} and from the classification of 2d discrete groups of isometries via their quotient spaces~\cite{Macbeath1967,Wilkie1966,thurston,Lucic1991,Conway2008}.   
Combinatorial tiling theory treats periodic (crystallographic) tilings of a simply connected space where the tiles are topological disks and defines an invariant called the Delaney-Dress symbol or  \emph{D-symbol} as a coloured weighted graph. 
From the D-symbol it is possible to reconstruct the tile shapes and adjacencies and the isomorphism class of the symmetry group for the tiling. 

In the $2$d setting, a D-symbol encodes a finite triangulation derived from the tiling. 
The underlying space of the triangulation is the quotient of the plane by a discrete group of isometries that preserve the tiling.  
This quotient space is a $2$-orbifold and can be viewed as a compact surface with (possibly) a finite number of boundary components and a finite number of isolated marked points.  
Although the full theory of D-symbols does not directly generalise to our setting of tilings by ribbons, the correspondence between crystallographic patterns and decorations of 2-orbifolds certainly does. 
The main issues to overcome are characterising the possible stabiliser subgroups for unbounded ribbon tiles, and constructing a triangulation from a given pattern. 
For Euclidean tilings by ribbons the possible stabiliser groups are the seven frieze groups; for the hyperbolic case, infinitely many non-euclidean frieze groups are also possible.   

Previous work related to this paper includes Huson's paper on tile-transitive partial tilings of the Euclidean plane~\cite{Husonpartial1993}, and the exploration of crystallographic line and tree patterns in the hyperbolic plane by Hyde et al.~\cite{Hyde2000,evansper1,evansper2,evansper3}. 
In this paper, we assume the tilings are locally finite and acted on by a discrete group of isometries of the plane with compact fundamental domain.
Definitions and notation for these groups of isometries are given in Section~\ref{groups}, and for combinatorial tiling theory in Section~\ref{tilings}. 
Results characterising the existence and structure of ribbon tiles are in Section~\ref{sec:ribbontiles}. 
Algorithms for enumerating and classifying tilings are described in Section~\ref{sec:gendsymbs}, with examples in Section~\ref{sec:ex}.

\section{Definitions and Preliminaries}
\label{sec:1}

\subsection{Groups of isometries}
\label{groups}

Let $\mathcal{X}$ be either the Euclidean ($\mathbb{E}^2$) or hyperbolic ($\mathbb{H}^2$) plane, and let $\Gamma$ be a discrete group of isometries of $\mathcal{X}$ having a compact fundamental domain.  If $\mathcal{X} = \mathbb{E}^2$ then $\Gamma$ is one of the 17 wallpaper groups of crystallography.  If $\mathcal{X} = \mathbb{H}^2$, then $\Gamma$ is a NEC group (non-Euclidean crystallographic group).  
We identify the isomorphism class of a group using Conway's \emph{orbifold symbol}~\cite{Conway2002}, a highly-readable version of Macbeath's group signature \cite{Macbeath1967}, as described below.  

For the purposes of this paper a $2$-orbifold, $\mathcal{O} = \mathcal{X}/\Gamma$, is a quotient space obtained by identifying points of $\mathcal{X}$ under the action of $\Gamma$.  That is, $x \sim y$ if $y = \gamma x$ for some $\gamma \in \Gamma$. 
The difference between $\mathcal{X}/\Gamma$ as a topological space and as an orbifold is that the full orbifold structure retains the metric information carried by the particular isometries of $\Gamma$ and an atlas of charts compatible with the $\Gamma$ action.  
We require both the topological view point and orbifold structure of the quotient space in this paper, and will use the script notation $\mathcal{O}$ for the orbifold with the additional structure and plain $O$ for its underlying topological space.   

It is well-known~\cite{thurston} that $2$-orbifolds have the topology of a finite-area $2$-manifold with a finite number of boundary components.  
Boundaries in a $2$-orbifold arise from the fixed lines of reflection isometries.
Other special points arise as the centres of rotational isometries; these are called \emph{cone points} if they lie in the interior of the orbifold, and \emph{corner points} if they lie on a boundary.  
The branching number, $N$, of a cone or corner point is the order of the rotational isometry, $\sigma$, that fixes that point i.e. $\sigma^N= id$. 
The boundaries, corner and cone points are collectively referred to as the \emph{singular locus} of the orbifold.  
The topology of a $2$-orbifold ($O$) is therefore specified by a symbol as follows: 
\begin{enumerate}
\item The number of handles, $h$, if the orbifold is orientable, or the number of cross-caps, $k$, if non-orientable.  Handles are denoted by \orb{\handle} at the beginning of the orbifold symbol.  Cross-caps are denoted by  \orb{\xcap} at the end of the orbifold symbol. 
\item The branching number for each cone point, listed in arbitrary order after any handles.  
\item The number of boundary components, $q$.  Each boundary component is represented by a \orb{\star} in the symbol.  Branching numbers for the corner points lying on each boundary component are listed in cyclic order, such that each boundary component has a consistent orientation for the manifold.  The ordering of the boundary components is arbitrary.  
\end{enumerate}
As simple examples, the group of isometries for a tiling of $\mathbb{E}^2$ by squares meeting four to a corner is  \orb{\star244}  ($p4m$ in Hermann-Mauguin notation),  
but for a Euclidean pattern with only translational symmetries it is \orb{\handle} ($p1$).   
 
It is known~\cite{thurston,Conway2002} that any orbifold symbol will correspond to a group of isometries of $\mathbb{E}^2, \, \mathbb{H}^2$, or $\mathbb{S}^2$, except for the symbols \orb{A}, \orb{\star A}, \orb{AB}, and \orb{\star AB}, with $\orb{A} \neq \orb{B}$. 
Moreover, the plane geometry associated with an orbifold can be deduced by computing a curvature-related quantity (the orbifold Euler characteristic) directly from the group symbol. 

Other discrete groups of isometries will be important when we discuss the internal symmetries of a tile as defined in the following section.  
For a bounded tile $T$, homeomorphic to a disk, the possible symmetry groups include that generated by a single reflection ($D_1$) and those that fix a single point.  The latter are well known as the cyclic and dihedral groups of order $N \geq 2$: 
\begin{align*}
	C_N & = \left<  q \;|\; q^N = id \right>  \\
	D_N & = \left < r_1, r_2 \;|\;  r_1^2 = id, \; r_2^2 = id, \; (r_1r_2)^N = id \right> 
\end{align*}
When these groups are viewed as acting on the tile $T$, we can form the quotient space $T/C_N$ or $T/ D_N$ and describe the topology of these quotient spaces similarly to the orbifold symbol above.  
This yields what Conway \emph{et al.}~\cite{Conway2008} call \emph{signatures} for rosette patterns:  \orb{N\bullet} and \orb{\star N \bullet} for the cyclic and dihedral groups respectively.  The symbol \orb{\bullet} represents for us a section of tile boundary in the quotient space.  
Note that the group generated by a single reflection isometry, $D_1 = \left< r \;|\; r^2 = id \right>$ is abstractly isomorphic to $C_2$, but its quotient space has the signature \orb{\star \bullet}. 

As we want to classify tilings by ribbons, we also need to consider, for example, isometries of a tile $T$ that is homeomorphic to $[0,1]\times\mathbb{R}$. 
The possible discrete symmetry groups for such a tile are $D_1$, $D_2$, $C_2$ or one of the seven frieze groups. 
Again, we can specify the quotient space structure of $T/G$ using a descriptive signature or symbol. 
In Table~\ref{tab:frieze_groups} we give the signatures (from \cite{Conway2008}) and Hermann-Mauguin (IUCr) name of each of the frieze groups.   In these signatures the \orb{\star} symbol still represents a single boundary component of $T/G$ and in our setting the  \orb{\infty} symbol represents a segment of tile boundary in $T/G$.  So the signature \orb{\star \infty \infty} implies that $T/G$ is a disk with a mirror boundary interrupted by two tile boundary segments and so combinatorially a quadrilateral.  Note that in other contexts, the \orb{\infty} symbol can represent an orbifold puncture as might be generated by a parabolic isometry of $\mathbb{H}^2$.  The different contexts are just other ways of obtaining a geometric realisation of the same abstract group.  

\begin{table}[h]  
\caption{\label{tab:frieze_groups} Signature of $T/G$ and IUCr name for the seven frieze groups, the index of the translation isometry in $G$, and a presentation that makes the translation isometry $t$ explicit in each case.} 
\begin{tabular}{cccl}
$T/G$  &   name  &  index & group presentation  \\
\hline  
 \orb{\infty \infty}  &  \emph{p1}  & 1 &  $\left< t \right>$ \\ 
 \orb{\infty x}  &  \emph{p11g}  &  2 & $\left< g, t \;|\;  g^2 = t \right>$  \\
 \orb{\infty \star}  &  \emph{p11m}   &  2 &  $\left< r, t \;|\;  r^2 , \,  rtr = t \right>$  \\
 \orb{\star \infty \infty}  & \emph{p1m1}  & 2 &  $\left< r_1, r_2, t  \;|\;  r_1^2 ,\,  r_2^2, \, r_1r_2 = t \right>$  \\
 \orb{22\infty}  &  \emph{p2}  &  2 & $\left< q_1, q_2, t  \;|\;  q_1^2 ,\,  q_2^2 , \, q_1q_2 = t \right>$  \\
 \orb{2\star\infty} & \emph{p2mg}  & 4 & $\left< r, q, t  \;|\;  r^2 ,\,  q^2, \, (qr)^2 = t \right>$  \\
 \orb{\star 22\infty}  &  \emph{p2mm}  & 4 & $\left< r_1, r_2, r_3, t  \;|\;  r_1^2 ,\,  r_2^2,\,  r_3^2 ,\,  (r_1r_2)^2, \, (r_2r_3)^2 , \, r_1r_3 = t \right>$  
 \end{tabular} 
 \end{table}

It is important to note that the above rosette and frieze groups can be realised using isometries of either the Euclidean or hyperbolic plane.  
If a rosette or frieze group, $G$, occurs as the subgroup of a wallpaper or NEC group, $\Gamma$, then it has infinite index, i.e. there are infinitely many cosets $\gamma G$.
The general question of characterising the subgroups of infinite index in NEC groups is covered in~\cite{Hoare1972,Hoare1973} where it is shown, as a special case, that if the subgroup has no reflection elements then it must be a free product of cyclic groups (of finite or infinite order). 
So, for example, \orb{22\infty} is simply the free product of two copies of $C_2$. 
The other frieze groups are given in Table~\ref{tab:frieze_presentations}.

\begin{table}[h]  
\caption{\label{tab:frieze_presentations} Frieze groups and their presentation as per results in~\cite{Hoare1973}.  The symbols $t$ and $g$ are generators of infinite order, with $t$ having a geometric interpretation as a translation and $g$ as a glide. The other generators use $r$ and $q$ which can be interpreted geometrically as reflections and rotations respectively.} 
\begin{tabular}{cl}
G  &   group presentation  \\
\hline  
 \orb{\infty \infty}  &     $C_{\infty} = \left< t \right>$ \\ 
 \orb{\infty x}  &    $C_{\infty} = \left< g  \right>$  \\
 \orb{\infty \star}  &    $C_2 \star C_{\infty} = \left< r, g \;|\;  r^2  \right>$  \\
 \orb{\star \infty \infty}  &  $C_2 \star C_2 =  \left< r_1, r_2  \;|\;  r_1^2 ,\,  r_2^2  \right>$  \\
 \orb{22\infty}  &     $C_2 \star C_2  = \left< q_1, q_2   \;|\;  q_1^2 ,\,  q_2^2 \right>$  \\
 \orb{2\star\infty} &   $C_2 \star C_2 = \left< r, q \;|\;  r^2, q^2 \right>$  \\
 \orb{\star 22\infty} & $C_2 \star C_2 \star C_2 = \left< r, q_1, q_2 \;|\; r^2, q_1^2, q_2^2 \right> $
 \end{tabular} 
 \end{table}
 
In the hyperbolic plane, we will also naturally encounter unbounded simply connected tiles with branching structure homeomorphic to a neighbourhood of an infinite tree embedded in $\mathbb{H}^2$.  
We call such tiles \emph{branched ribbons}, examples are given in Section~\ref{sec:ex}. 
The isometries of such a tile are isomorphic to a group action on a tree, and are covered by the theory of Bass-Serre~\cite{stilwell2002trees}.  
The simplest examples of group actions on trees are those that have a line segment as fundamental domain.
These groups are a free product with amalgamation of the subgroups that fix the vertices, amalgamated via the subgroup that fixes the edge (the oriented line segment).  
For example, if the line segment generating the tree has end points on rotation centers of order \orb{A} and \orb{B}, and the edge group is trivial, then the group is $G = C_A \star C_B$ and $T/G$ has the signature \orb{AB\infty}.

\subsection{Combinatorial tiling theory}
\label{tilings}

\emph{Combinatorial tiling theory} describes the combinatorial structure of tilings of a simply connected space for which each tile is homeomorphic to a closed and bounded disk. 
The definitions below follow those given for two-dimensional spaces in \cite{Huson1993}. 
A set $\mathcal{T}$ of topological disks in $\mathcal{X}$ is called a \emph{tiling} if every point  $x\in\mathcal{X}$ belongs to at least one tile $T\in \mathcal{T}$ and every two tiles $T_1$ and $T_2$ of $\mathcal{T}$ have disjoint interior.  
All tilings in this paper will be assumed to be locally finite, i.e. any compact disk in $\mathcal{X}$ meets only a finite number of tiles. 
The vertices and edges of a tile are defined topologically rather than using the geometry of straight lines and corners.  
So, a \emph{vertex} is a point that is contained in at least three tiles, and 
an \emph{edge} is a connected segment of tile boundary joining two vertices. 

Let $\mathcal{T}$ be a tiling of $\mathcal{X}$ and let $\Gamma$ be a discrete group of isometries.  
If  $\mathcal{T}=\gamma \mathcal{T} :=\{\gamma T \;|\; T\in\mathcal{T}\}$ for all $\gamma \in \Gamma$ then we call the pair $(\mathcal{T},\Gamma)$ an \emph{equivariant tiling}.  
Two tiles $T_1,T_2\in \mathcal{T}$ are \emph{equivalent} or symmetry-related if there exists $\gamma\in\Gamma$ such that $\gamma T_1=T_2.$  
The \emph{orbit} of a tile is the subset of $\mathcal{T}$ given by images of $T$:  $\Gamma.T = \{ \gamma T  $ for $ \gamma \in \Gamma \}$. 
Given a particular tile $T\in\mathcal{T}$, the \emph{stabilizer subgroup} $\Gamma_T$ is the subgroup of $\Gamma$ that fixes $T$, i.e. $\Gamma_T = \{ \gamma \in \Gamma \;|\; \gamma T = T \}$.  
A tile is called \emph{fundamental} if $\Gamma_T$ is trivial and we call the whole tiling fundamental if this is true for all tiles. 
An equivariant tiling is called \emph{tile}-$k$-\emph{transitive}, when $k$ is the number of equivalence classes (i.e. distinct orbits) of tile under the action of $\Gamma$.  
We can also study the action of $\Gamma$ on tile edges and vertices and define edge- and vertex-$k$-transitivity similarly.  
Note that the above definitions do not require $\Gamma$ to be the maximal symmetry group for the tiling $\mathcal{T}$.

Two tilings $(\mathcal{T}_1, \Gamma_1)$ and $(\mathcal{T}_2, \Gamma_2)$ of a simply connected space $\mathcal{X}$ are \emph{equivariantly equivalent} if there is a homeomorphism, $\phi$, of $\mathcal{X}$ such that $\phi(T_1) \in \mathcal{T}_2$ for all $T_1 \in \mathcal{T}_1$ and such that $\phi$ induces a group isomorphism of $\Gamma_1$ onto $\Gamma_2$ by $\Gamma_2=\phi \Gamma_1 \phi^{-1}.$  
A natural question is  whether there is an invariant that detects when two tilings are equivariantly equivalent.  
Dress \emph{et al.},~\cite{DRESS1987} show that a complete invariant is indeed possible for tilings of simply connected manifolds.  
The invariant, called the D-symbol, consists of a graph that records adjacencies between tiles and their faces, augmented by weights that encode the group action of $\Gamma$ on $\mathcal{T}$.   
The D-symbol can be interpreted as encoding a simplicial structure on an orbifold, obtained from barycentric subdivision of the tiling.  
For the 2d case this means a triangulation of a 2-orbifold where each 2-simplex spans a tile centre-point, edge mid-point and tiling vertex.  
This structure is exploited in~\cite{Huson1993} to achieve a fully algorithmic approach to the enumeration and identification up to equivariant equivalence of 2d tilings of $\mathbb{S}^2, \mathbb{E}^2$ and $\mathbb{H}^2$.  

A fundamental tile-1-transitive equivariant tiling $(\mathcal{T}, \Gamma)$ has a single type of tile, $T_0$, that is a fundamental domain for $\Gamma$. 
Conversely, any fundamental domain for $\Gamma$ homeomorphic to a disk also gives rise to such a tiling. 
As established by Wilkie in~\cite{Wilkie1966} such a tiling corresponds to a presentation for $\Gamma$, with each edge of $T_0$ corresponding to a generator and each vertex to a relation.  
If $e_i = T_0 \cap T_i$, then the edge-generator is an element, $[e_i]  \in \Gamma$, such that $T_i = [e_i] T_0$. 
The relation at vertex $v$ is found by listing the generators associated with each edge crossing when making a clockwise circuit around $v$.   See~\cite{Wilkie1966} for further details. 

The papers by Lucic \emph{et al.}~\cite{Lucic1990,Lucic1991,Lucic2018} present a method for enumerating the different combinatorial forms of disk-like fundamental domains for crystallographic groups. 
In particular, the authors establish that vertices and edges of a tiling $(\mathcal{T}, \Gamma)$ map via the group action onto vertices and edges of a graph $C = (V,E)$ embedded in the quotient space $O = \mathcal{X}/\Gamma$. 
In general, vertices of $C$ come from vertices of $T$, except in the case that an edge midpoint in $T$ is a cone point of order 2.  Such an edge of $T$ maps onto an edge of $C$ with a vertex of degree 1 on the cone point. 
The graph $C$ embedded in $O$ has the following properties: 
\begin{enumerate}
\item  $O\setminus C$ is an open disk. 
\item  Each cone point is a vertex of $C$ with at least one incident edge in $C$.   
\item  Each (mirror) boundary component of $O$ lies in a subgraph of $C$.  
\item  Each corner point is a vertex of $C$ with at least two incident edges. 
\item  Let $\tilde{O}$ be the closed surface of genus $g$ obtained from $O$ by capping each boundary component of $O$ with a disk.  
Then $C$ is contractible in $\tilde{O}$ to the graph $\tilde{C}$ with one vertex and $2g$ loops if $O$ is orientable and $g$ loops if $O$ is non-orientable.  
In particular, each subgraph $C_i$ that belongs to the $i$-th boundary component of $O$ is a contractible loop in $\tilde{O}$.
\item Any vertex of $C$ that does not lie on a cone point, corner point or boundary must have at least three incident edges in $C$.    
\end{enumerate}
See \cite{Lucic1991,Lucic2018} for representative general figures, and figure \ref{fig:EuclideanEx} for a simple Euclidean example. 

Another approach to enumerating fundamental tile-1-transitive tilings using a combinatorial requirement on D-symbols is given by Huson in~\cite{Huson1993}.  
The same paper then shows that tile-$k$-transitive fundamental tilings can be derived from those of transitivity $(k-1)$ using the operation of \emph{tile splitting} and non-fundamental tile-$k$-transitive tilings are obtained from fundamental ones by \emph{tile glueing}.   We describe these operations below in terms of modifications to $C$, the graph of tile edges in $O$. 

Given a tile-$k$-transitive fundamental tiling, let $C$ be its corresponding graph on $O$.  
Tile-splitting adds a new segment $e$ to $C$ such that $O\setminus (C\cup{e})$ is the union of $k+1$ disks.

Now, suppose we have a tile-$k$-transitive tiling with $j \leq k$ classes of fundamental tile.  
A fundamental tile is identified as a component of $O\setminus C$ that contains no part of the singular locus of $\mathcal{O}$.  
Tile glueing erases at most two edges from $C$ to get $C'$ so that $C'$ is connected, $O\setminus C'$ is still the union of $k$ disks, and the tiling has $(j-1)$ classes of fundamental tile. 
The edges of $C$ that can be erased must be incident only to one transitivity class of fundamental tile and are of three types:
\begin{enumerate}
\item An edge of $C$ that has a vertex of degree 1 at a cone point of order \orb{N}. 
This glues $N$ copies of a fundamental tile into one new one with stabiliser group \orb{N\bullet}. 
\item A pair of edges of $C$ that lie in a mirror boundary and meet at a vertex of degree 2 on a corner point of order \orb{N}. This glues $2N$ copies of a fundamental tile into one new one with stabiliser group \orb{\star N\bullet}.  
\item A single edge of $C$ that is a segment of mirror boundary and has vertices of degree at least 3 (or degree 2 on a corner point).  This glues two copies of a tile together into one with stabiliser group \orb{\star \bullet}. 
\end{enumerate}
As discussed in~\cite{Huson1993}, a sequence of tile splits and glues can sometimes lead to a tile whose interior is a disk, but whose closure is not.  Such tiles can be identified algorithmically from the combinatorics of the D-symbol as described in that paper.

In the following sections, we develop the mathematical groundwork required to characterise edge erase operations that lead to unbounded ribbon and branched-ribbon tilings.  

\section{Orbifold paths and tile glueing} 
\label{sec:ribbontiles}

We will study the relationship between closed paths in the orbifold $\mathcal{O} = \mathcal{X}/\Gamma$, their lifts in $\mathcal{X}$ (i.e. the Euclidean or hyperbolic plane), and equivariant tilings that contain unbounded tiles.  
Firstly: 
\begin{definition}
A (possibly branched) \emph{ribbon tiling} of $\mathcal{X}$ is a countable set $\mathcal{T}$ of connected closed domains, $T_i$,  such that every point $x\in \mathcal{X}$ belongs to at least one tile, all tiles have pairwise disjoint interiors, and such that any bounded disk in $\mathcal{X}$ intersects finitely many tiles.  
\end{definition}
An equivariant ribbon tiling is one that is mapped to itself by a discrete group of isometries of $\mathcal{X}$. A simple example of a ribbon tiling that is not an example of a classical tiling is the trivial tiling, where there is only one tile that is all of $\mathcal{X}$. 
Other examples are the partial tilings enumerated by Huson in~\cite{Husonpartial1993}, if we treat his complementary regions as ribbon tiles. 

Below, we denote the singular locus of $\mathcal{O}$ (and $O$) by $\Sigma$, and the projection map as $p : \mathcal{X} \to O$.   
\begin{definition}
An \emph{orbifold loop} in $\mathcal{O}$ is built from a sequence of paths $\{\alpha_i\}_{i=1}^k : [t_{i-1}, t_{i}] \to \mathcal{X}$ 
where $0 = t_0 < t_1 < \cdots < t_k = 1$ such that the $\alpha_i$ project to a loop in $O$, i.e. 
$p(\alpha_i(t_i))= p(\alpha_{i+1}(t_i)) \: \forall i\in \{1,...,k-1\}$ and $p(\alpha_k(t_k)) = p(\alpha_{1}(0))$. 
Furthermore, we require each segment $\alpha_i([t_{i-1},t_{i}])$ to be contained in a single orbifold coordinate patch, and attach to each $t_i$, $i=1,\ldots,k$, an element $\gamma_i \in \Gamma$ that corresponds to a coordinate change map from the patch containing $\alpha_{i}$ to that containing $\alpha_{i+1}$, with $\gamma_k$ being a coordinate change from $\alpha_{k}$ to $\alpha_0$. 
\end{definition}
The group elements $\gamma_i$ that lift the coordinate changes at the points where the $\alpha_i$ fit together let us  distinguish the situation where an orbifold path crosses a mirror line in $\mathcal{X}$ ($\gamma$ is the reflection) or simply backtracks after touching it ($\gamma$ is the identity). 

We also define a \emph{simple} orbifold loop as one that has no self-intersections in $\mathcal{O}$ and does not pass through any cone or corner points.  

Two orbifold loops are homotopic if the sequence of paths in $\mathcal{X}$ are $\Gamma$-equivariantly homotopic. Note that we explicitly allow for concatenation and splitting of paths, where applicable, i.e. the number of segments, $k$, in the definition above is not necessarily fixed during a homotopic deformation.  Also refer to \cite{ratcliffe2006foundations} for further details.

The \emph{orbifold fundamental group} $\pi_1^{orb} (\mathcal{O}, x)$, $x \in \mathcal{O}\setminus \Sigma$, can be defined as the set of orbifold loops based at $x_0 \in \mathcal{X}$ up to homotopy equivalence in $\mathcal{X}$, where $x_0 \in p^{-1}(x)$. 
In practice, one can interpret this to mean that a loop in $\mathcal{O}$ that winds once around a cone point of order $N$ cannot be contracted to the point $x$, but a closed path that winds exactly $N$ times around a neighbourhood of  this cone point (and avoids the rest of $\Sigma$) is null-homotopic. 
Note that the contractibility of an orbifold loop cannot be seen simply by looking at the corresponding image curve in $O$, in contrast to the classical situation for manifolds.

It is known that for our geometric 2-orbifolds, $\pi_1^{orb}(\mathcal{X}/\Gamma) \simeq \Gamma$ in the natural way.  In other words, $\pi_1^{orb}(\mathcal{X}/\Gamma)$ is the group of deck transformations for the \emph{branched} covering map $p: \mathcal{X} \to \mathcal{O}$.  For proofs and a more detailed account, refer to \cite[chapter $13$]{ratcliffe2006foundations}. 

\bigskip

We now study what happens when we erase a subset of edges from a fundamental-domain tiling. 
Recall that the edges of a tiling map to an embedded graph $C= (V,E)$ in $O$.  
Let $S = \{e_1, \ldots, e_k\} \subset E$ be the edges to be erased and $R = E \setminus S$ be the edges that remain. 
Let $C_R = (V_R, R) \subset C$ be the subgraph obtained from $C$ by erasing $S$ and any vertices left isolated. We also want to avoid ``dangling ends'', so for all $e \in R$ with a vertex $v$ of degree-1 in $C_R$ with $v \notin \Sigma$ we add $e$ to $S$.  

\begin{theorem}\label{thm:edgedel}
Let $\mathcal{X} = \mathbb{E}^2$ or $\mathbb{H}^2$ and suppose $\Gamma$ is a wallpaper or NEC group 
Suppose we are given a fundamental tile-1-transitive tiling $(\mathcal{F}, \Gamma)$ whose edges map onto a graph $C$ embedded in $O = \mathcal{X}/\Gamma$. 
Let $S = \{e_1,\ldots,e_k\}$ be a subset of edges of $C$ whose removal avoids dangling ends.  
Then erasing all preimages of these edges from $\mathcal{F}$ results in a non-fundamental tile-1-transitive ribbon tiling $(\mathcal{T}, \Gamma)$, such that the stabiliser group of each tile $T \in \mathcal{T}$ is isomorphic to the subgroup of $\Gamma$ generated by the erased edges. 
\end{theorem}

\begin{proof}
The correspondence between edges of a fundamental domain, $F \subset \mathcal{X}$, and generators for a presentation of $\Gamma$ are established in~\cite{Wilkie1966}, as summarised in Section~\ref{tilings}.  If $f$ is an edge of $F$, then $f = F \cap \gamma(F)$ for an element $\gamma \in \Gamma$, and we use the notation $[f]$ for this group element.  
We also know that each edge $f$ that is not a segment of mirror boundary is glued to another edge $f' \in F$ (possibly itself) so that $[f'] = [f]^{-1}$.  
The image of $f$ in $O$ is an edge of $C$, $p(f) = p(f') = e$. 

Choose a point $x \in \inter (F)$, and for each $e_i \in C$ choose a single tile edge $f_i \in p^{-1}(e_i) \cap \cl (F)$. 
Then for each $f_i$, there is a simple orbifold loop $\alpha_{e_i}$ with $\alpha_{e_i}(0) = x$ and $\alpha_{e_i}(1) = [f_i](x)$, such that $\alpha_{e_i}([0,1])$ is a connected curve in $\mathcal{X}$ that intersects the boundary of $F$ in a single point of $f_i$. 
This follows from the fact that $F$ is a topological disk.  
In the deck-transformation correspondence between the orbifold fundamental group and $\Gamma$, we then have that $[\alpha_{e_i}] \sim [f_i]$.  

Now let $H$ be the subgroup of $\Gamma$ generated by the group elements associated with edges $e_i \in S$ and let $T = \bigcup_{\eta \in H} \eta (F)$.  
$T$ is path connected by the following argument.  
Each orbifold loop $\alpha_{e_i}$ has a connected representative from $x \in F$ to $[f_i](x)$ in $\mathcal{X}$, and 
any other such connected representative of $\alpha_{e_i}$ has the end-points $\gamma(x)$ and $\gamma[f_i](x)$ for some $\gamma \in \Gamma$.  
Therefore, for each $\eta \in H$, there is a path in $\mathcal{X}$ from $x$ to $\eta(x)$ that lies entirely within $T$, obtained by writing $\eta$ as a word in $[f_1], \ldots, [f_k]$, and forming the corresponding concatenation of the lifts of the $\alpha_{e_i}$ loops.  

By their definitions, $H$ is the stabiliser subgroup of the ribbon tile $T$.  
If $H$ is finite, then it must be a cyclic or dihedral group, the erased edges must be one of the three cases discussed in Section~\ref{tilings}, and $T$ will be bounded.  
If $H$ is infinite and a proper subgroup, then it is a free product of groups as described in~\cite{Hoare1972,Hoare1973} and $T$ must be unbounded as it is the union of infinitely many distinct copies of $F$.  
The case that $H = \Gamma$ means $T$ is the whole of $\mathcal{X}$.  

Next consider the action of an isometry, $\gamma \in \Gamma$ on a glued tile $T$.  
Recall that the construction of $T$ began with a particular choice of fundamental domain tile $F \subset \mathcal{X}$,  
and that $\Gamma$ acts transitively on the tiling $\mathcal{F}$. 
If $\gamma \in H$, then $\gamma F \in T$, $\gamma \eta \in H$ for all $\eta \in H$ and so $\gamma T = T$. 
If $\gamma \notin H$, then $\gamma F \notin T$, and in particular, $\gamma \eta F \notin T$ for any $\eta \in H$, so that $\gamma (\inter(T)) \cap \inter(T) = \emptyset$.  
It follows that for any two $\gamma, \gamma' \notin H$, that either $\gamma T = \gamma' T$ or $\gamma (\inter(T)) \cap \gamma' \inter(T) = \emptyset$.   
So let $\mathcal{T}$ be the union of all distinct images of $T$.  
It then follows from the tile-transitivity of $(\mathcal{F},\Gamma)$, that $(\mathcal{T}, \Gamma)$ is also a tile-transitive ribbon tiling of $\mathcal{X}$.  
$\qed$
\end{proof}

\begin{lemma}\label{lemma:infstab}
The stabiliser subgroup, $H$, for the non-fundamental tile $T$, generated by erasing edges as in Theorem~\ref{thm:edgedel} is infinite if and only if it contains a translation isometry. 
\end{lemma}
\begin{proof}
First note that if $H$ contains a translation isometry then it must be infinite.  And recall that a glide applied twice is a translation.  

Now assume $H$ is an infinite group of isometries of $\mathbb{H}^2$. 
If $H$ is abelian, then it must be generated by a single isometry of infinite order, i.e., a translation, glide or parabolic rotation (an isometry with a single fixed point at infinity).  This last case can be ruled out as our tiling group $\Gamma$ is assumed to have a compact orbifold. 
If $H$ is non-abelian, then we have a result from~\cite{Friedrich2008} that such a group must contain a translation. 

Finally, we consider the case that $H$ is an infinite discrete subgroup of isometries of $\mathbb{E}^2$: 
$H \subset O(2) \rtimes \mathbb{R}^2$. 
If $H$ consists entirely of rotations and/or reflection isometries, then it must be a discrete subgroup of $O(2)$, and therefore finite. 
Since we assume $H$ is infinite, not all its elements can be rotations and reflections and we see that it must contain a translation or glide. 
$\qed$ 
\end{proof}

\begin{lemma}\label{thm:simpleconn}
Let $(\mathcal{T},\Gamma)$ be a  tiling obtained via edge deletion from a fundamental tile-1-transitive tiling as in Theorem~\ref{thm:edgedel}.  If the stabiliser group $H$ is infinite, then the tile $T$ is simply connected: it is a ribbon or branched ribbon. 
\end{lemma}
\begin{proof}
From Theorem~\ref{thm:edgedel}, it follows that $\mathcal{X}$ is the union of path-connected, unbounded tiles of the form $\gamma T$ for $\gamma \in \Gamma$ and that all of these tiles have disjoint interiors. 
If $T$ is not simply connected, then it bounds a region of $\mathcal{X}$ that is covered by isometric copies of $T$, which is clearly a contradiction. 
$\qed$
\end{proof}

\bigskip 
 
We briefly discuss the tile-$k$-transitive case. By the results of section \ref{tilings}, every tile-$k$-transitive fundamental tiling comes from splitting an initial fundamental tile-$1$-transitive tile into $k$ pieces, each homeomorphic to a disk. 
Now suppose that $0 < j \leq k$ tiles are fundamental, and consider what edges may be erased so that we obtain a tiling with $(j-1)$ fundamental tiles.  
We want to stay within the class of tile-$k$-transitive tilings, so the allowed deletions are restricted to edges that are incident to only one symmetry class of fundamental tile, as required in the classical setting. 
The construction of the glued tile, $T$, and its stabiliser group, $H$, then proceeds in the same manner as for the tile-1-transitive case discussed above. 
Each deleted edge has a simple orbifold loop associated to it avoiding all other edges, so the resulting tile is path-connected, but not necessarily homeomorphic to a disk, nor simply connected.  
If $H$ is finite, then the topology of $\cl(T)$ may be that of a closed disk (as before) or a closed disk with finitely may holes.  
If $H$ is infinite, then it can have finite or infinite index in $\Gamma$.  
The first case occurs only if $H = \Gamma$. Indeed, the existence of even a single edge in the graph $C$ embedded in $\mathcal{X}/\Gamma$ results in infinitely many connected edge components in the corresponding tiling in $\mathcal{X}$. We call the tiling in this case ``patchy.'' 
That is, the tile with stabiliser group $H$ may have bounded holes containing other tiles; Euclidean examples of these are enumerated in~\cite{Husonpartial1993}. 
In the case $H$ has infinite index in an NEC group, it must be a free product of groups of the form in~\cite{Hoare1973}. 
The Euclidean case is called ``stripey'' in~\cite{Husonpartial1993}: the complementary region in those examples is a tile with infinite stabiliser group and so a ribbon.

\section{Enumeration and classification of branched ribbon tilings}
\label{sec:gendsymbs}

We are now in a position to enumerate crystallographic tilings with ribbon tiles. The steps are as follows: 
\begin{enumerate}
\item Select a symmetry group of interest, and construct its orbifold.
\item Enumerate the possible tile-1-transitive fundamental tilings with methods described in \cite{Lucic1991,Lucic2018} or \cite{Huson1993}, and represent these as graphs embedded in the orbifold.  
\item Systematically delete subsets of edges from the embedded graphs as described in the section above to derive all tile-1-transitive non-fundamental tilings, both regular ones with bounded tile and ribbon ones with unbounded ribbon or  branched-ribbon tiles. 
\item To enumerate tile-2-transitive tilings, apply Huson's SPLIT algorithm to the fundamental-domain tile, then systematically delete allowed sets of edges incident to one tile type, and then the second tile type. 
\item To build more complex examples, apply successive split operations to fundamental tiles, followed by edge erasing or tile glueing on each type of tile.
\end{enumerate}

Steps 1-3 are illustrated for the wallpaper group \emph{pmg} with Euclidean orbifold \orb{22\star} in figures~\ref{fig:EuclideanEx}--\ref{fig:EuclideanEx4}.  

\begin{figure}[!tbp]
  \begin{subfigure}[t]{0.3\textwidth}
  \centering
    \includegraphics[width=\textwidth]{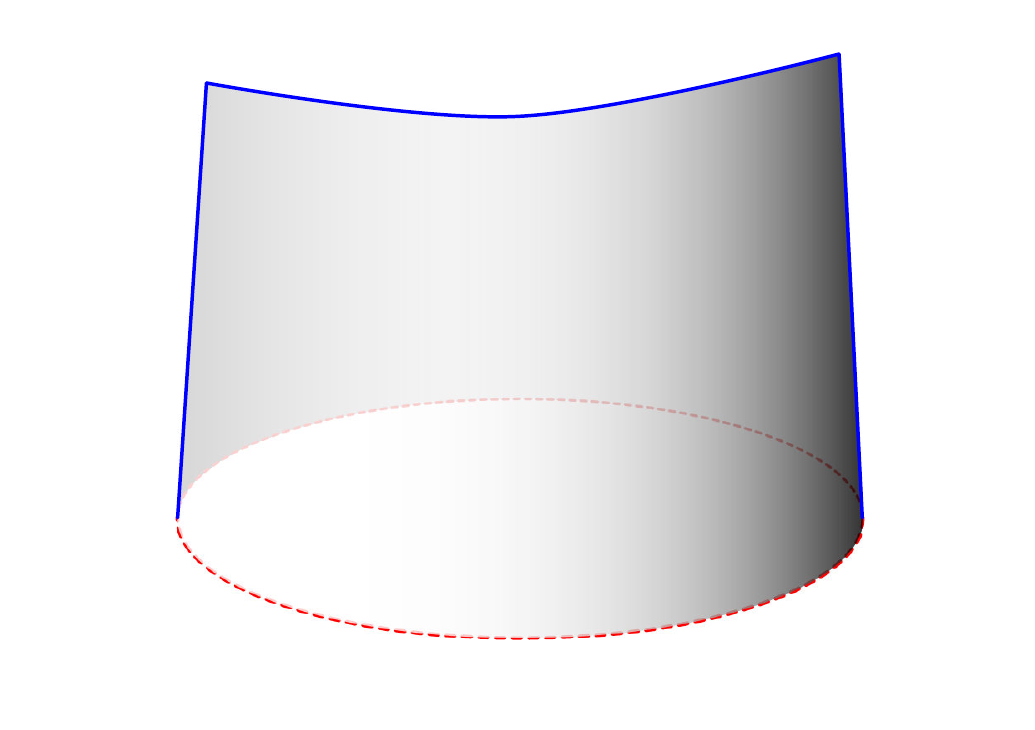}
    \caption{The Euclidean orbifold \orbcap{22\star} which is a sphere with one boundary curve (the red dotted line) and two cone points of order 2.}\label{subfig:22star_orb}
  \end{subfigure}
  \hfill
  \begin{subfigure}[t]{0.3\textwidth}
  \centering
    \includegraphics[width=\textwidth]{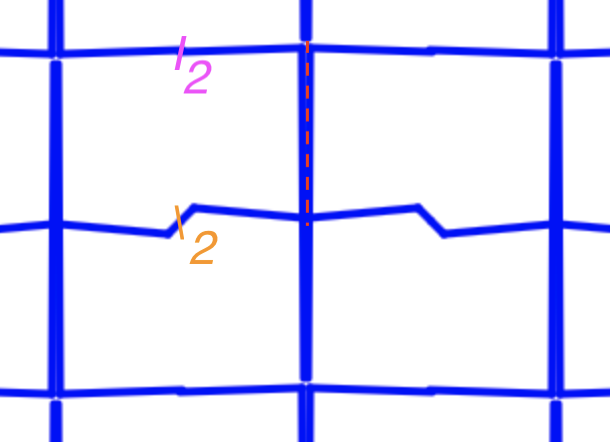}
  \caption{A slightly deformed standard fundamental domain for \orbcap{22\star}.  Four fundamental domains make up the rectangular translational unit cell shown.}\label{subfig:22star_FD_st}
  \end{subfigure}
    \hfill
  \begin{subfigure}[t]{0.3\textwidth}
  \centering
    \includegraphics[width=\textwidth]{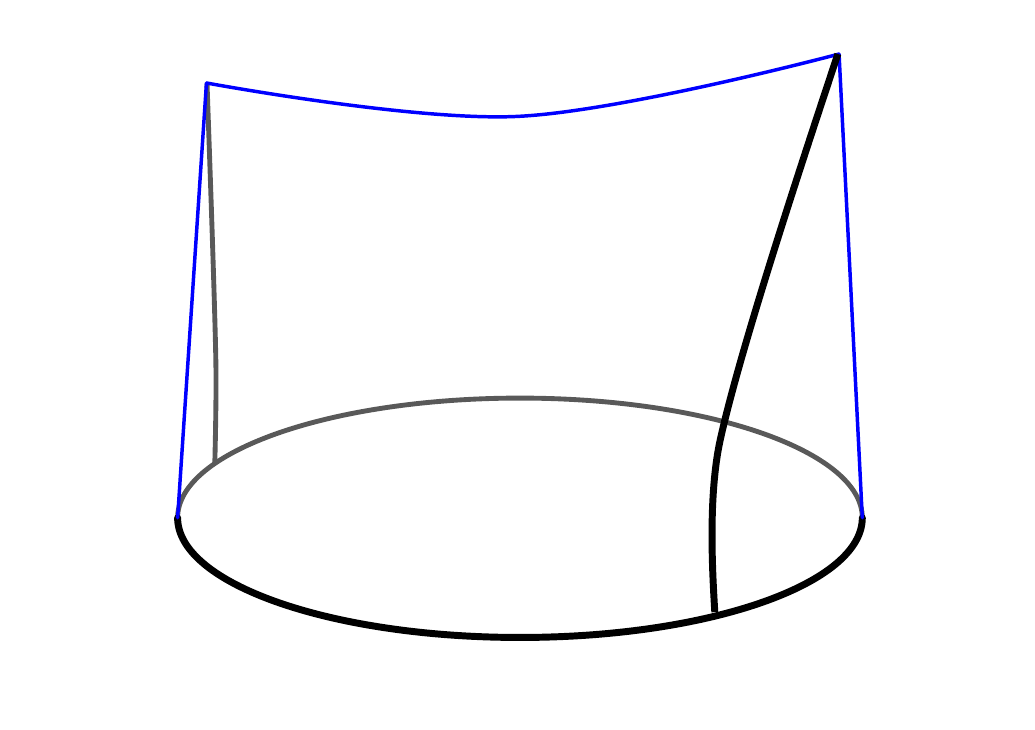}
  \caption{The standard fundamental domain drawn as a graph embedded in the orbifold quotient space.}\label{subfig:22star_orb_st}
  \end{subfigure}
 \caption{An example of a Euclidean orbifold quotient space, its geometric realisation and edge graph diagram.  The vertical lines of (b) lie along the mirrors and the two inequivalent centres of rotation are marked. }\label{fig:EuclideanEx}
\end{figure}

\begin{figure}[!tbp]
  \begin{subfigure}[t]{0.3\textwidth}
  \centering
    \includegraphics[width=\textwidth]{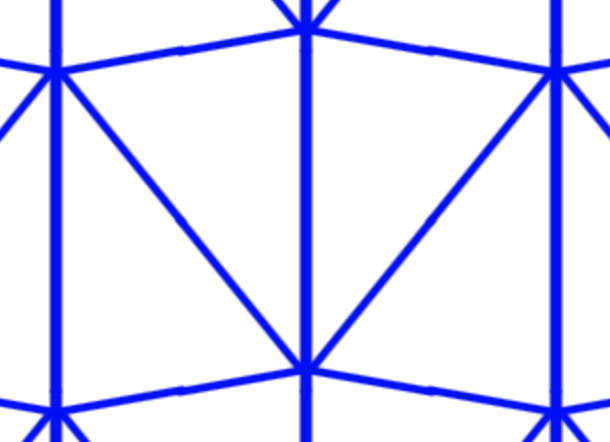}
    \caption{ }\label{subfig:22star_a}
  \end{subfigure}
  \hfill
  \begin{subfigure}[t]{0.3\textwidth}
  \centering
    \includegraphics[width=\textwidth]{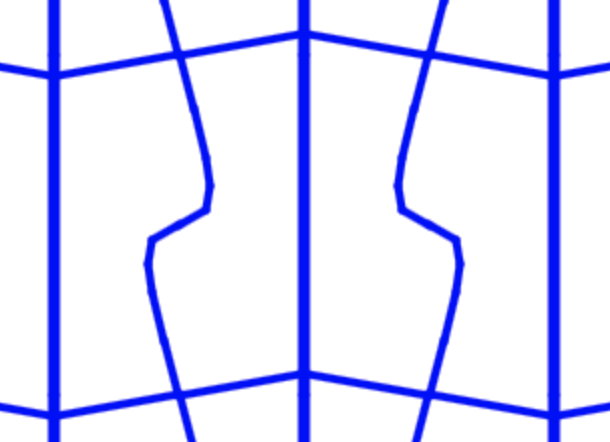}
  \caption{ }\label{subfig:22star_b}
  \end{subfigure}
    \hfill
  \begin{subfigure}[t]{0.3\textwidth}
  \centering
    \includegraphics[width=\textwidth]{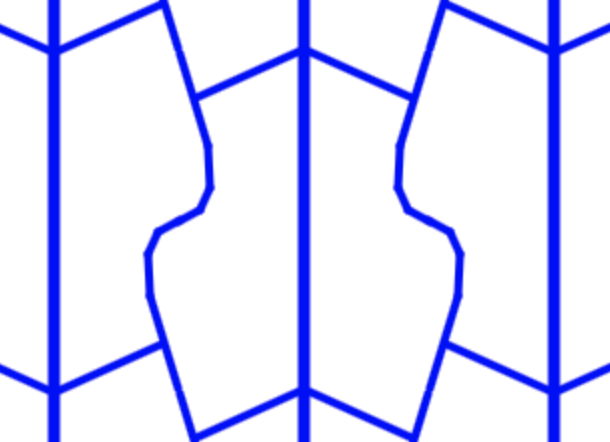}
  \caption{ }\label{subfig:22star_c}
  \end{subfigure}
  \\
    \begin{subfigure}[t]{0.3\textwidth}
  \centering
    \includegraphics[width=\textwidth]{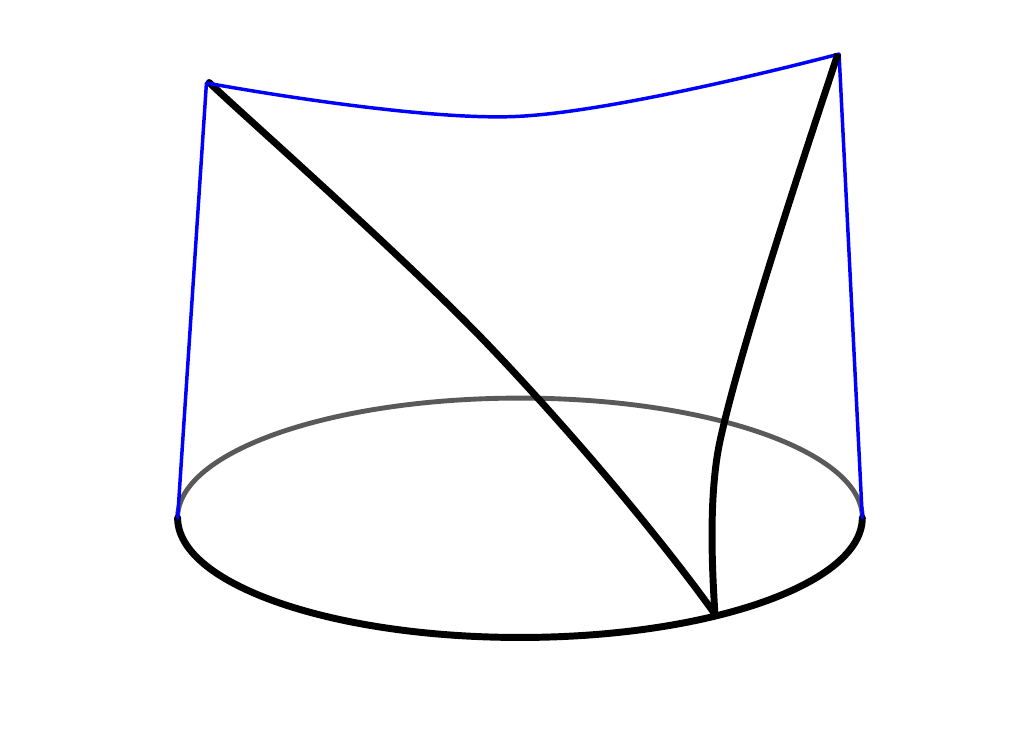}
    \caption{ }\label{subfig:22star_d}
  \end{subfigure}
  \hfill
  \begin{subfigure}[t]{0.3\textwidth}
  \centering
    \includegraphics[width=\textwidth]{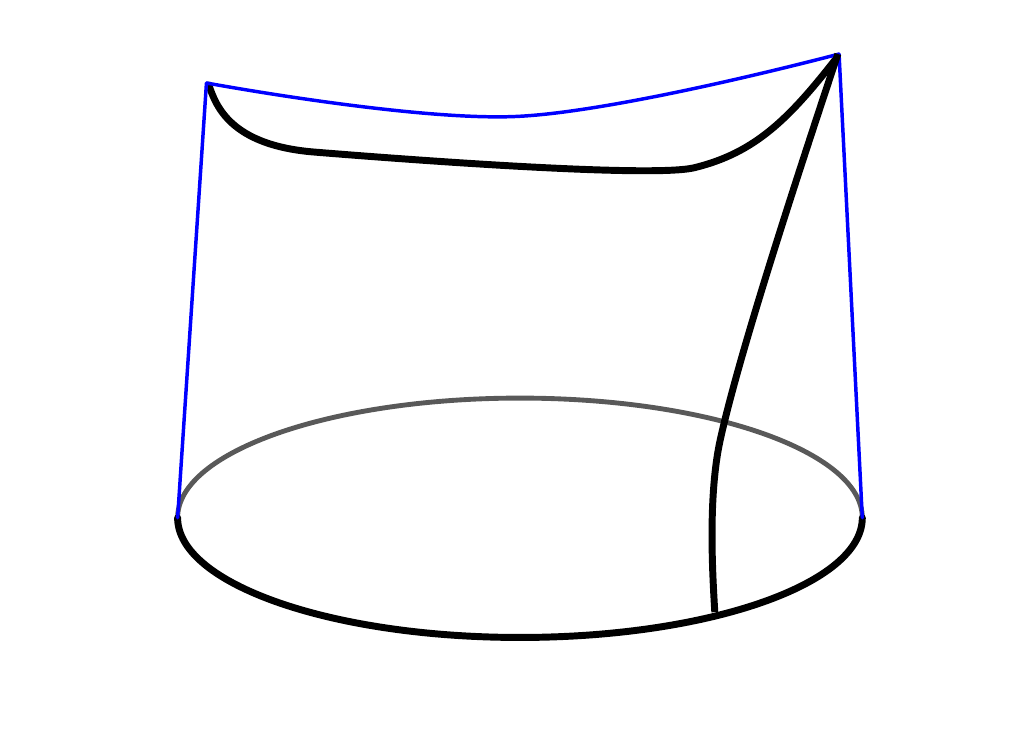}
  \caption{ }\label{subfig:22star_e}
  \end{subfigure}
    \hfill
  \begin{subfigure}[t]{0.3\textwidth}
  \centering
    \includegraphics[width=\textwidth]{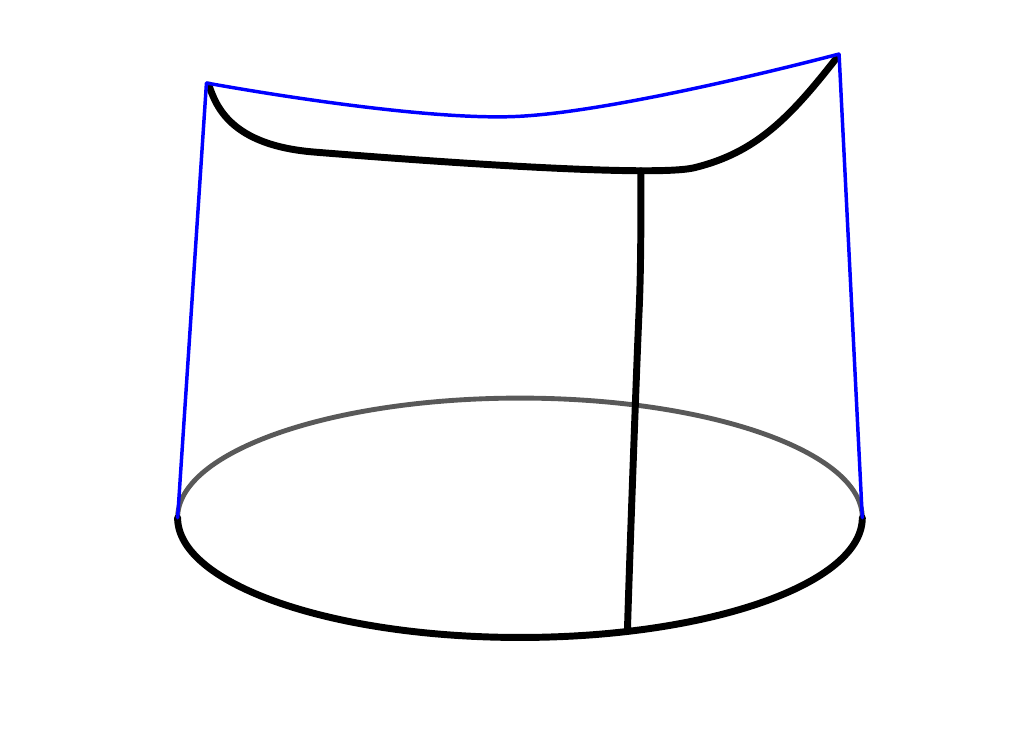}
  \caption{ }\label{subfig:22star_f}
  \end{subfigure}
 \caption{The other combintorial types of fundamental domain for \orbcap{22\star} with (a) triangular, (b) quadrilateral and (c) pentagonal polygon regions. (d)--(f) The corresponding edge graphs embedded in the orbifold quotient space. }\label{fig:EuclideanEx2}
\end{figure}

\begin{figure}[!tbp]
  \begin{subfigure}[t]{0.3\textwidth}
  \centering
    \includegraphics[width=\textwidth]{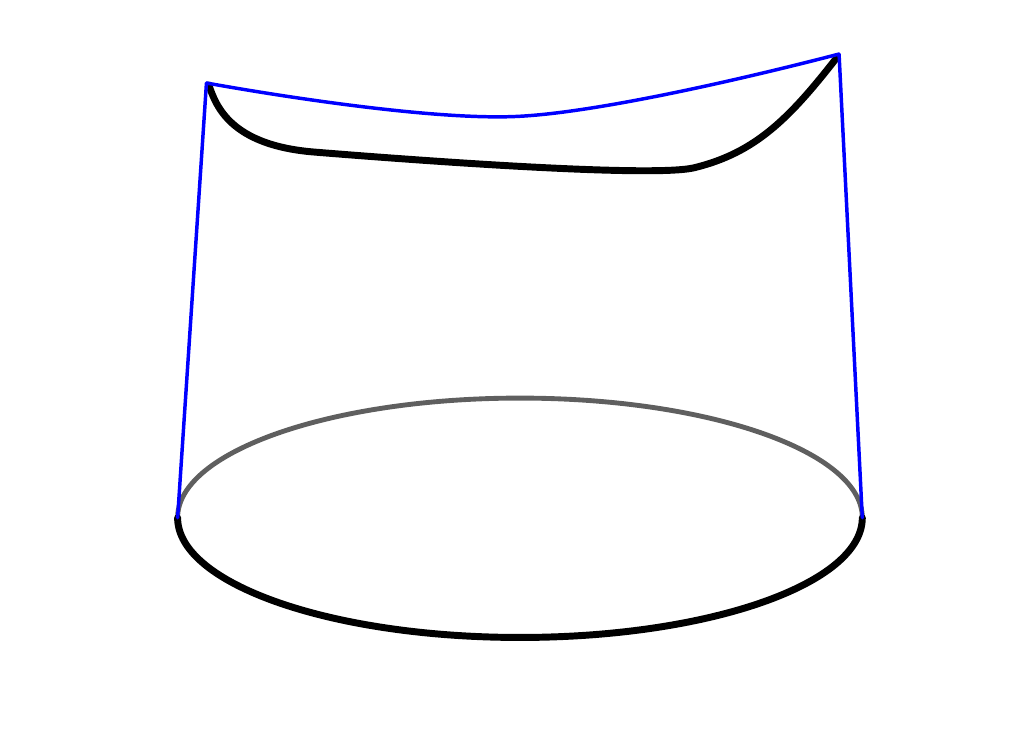}
    \caption{Ribbon tile with stabiliser \orbcap{\infty\infty}.}\label{subfig:22star_ra_orb}
  \end{subfigure}
  \hfill
  \begin{subfigure}[t]{0.3\textwidth}
  \centering
    \includegraphics[width=\textwidth]{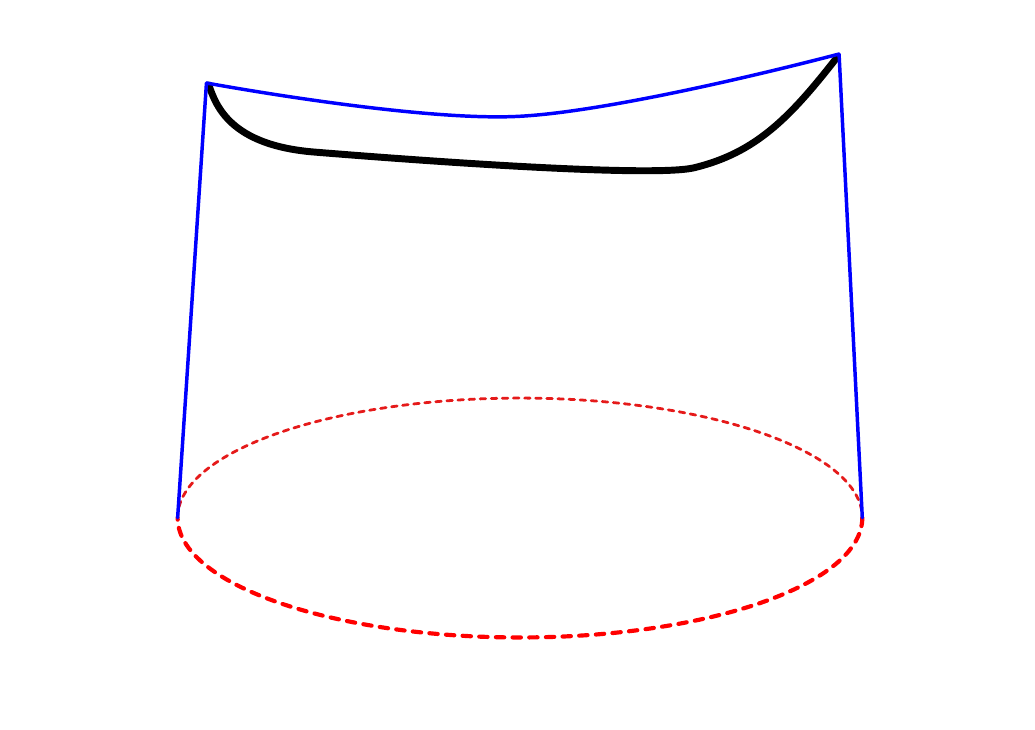}
  \caption{Ribbon tile with stabiliser \orbcap{\infty\star}. }\label{subfig:22star_rb_orb}
  \end{subfigure}
    \hfill
  \begin{subfigure}[t]{0.3\textwidth}
  \centering
    \includegraphics[width=\textwidth]{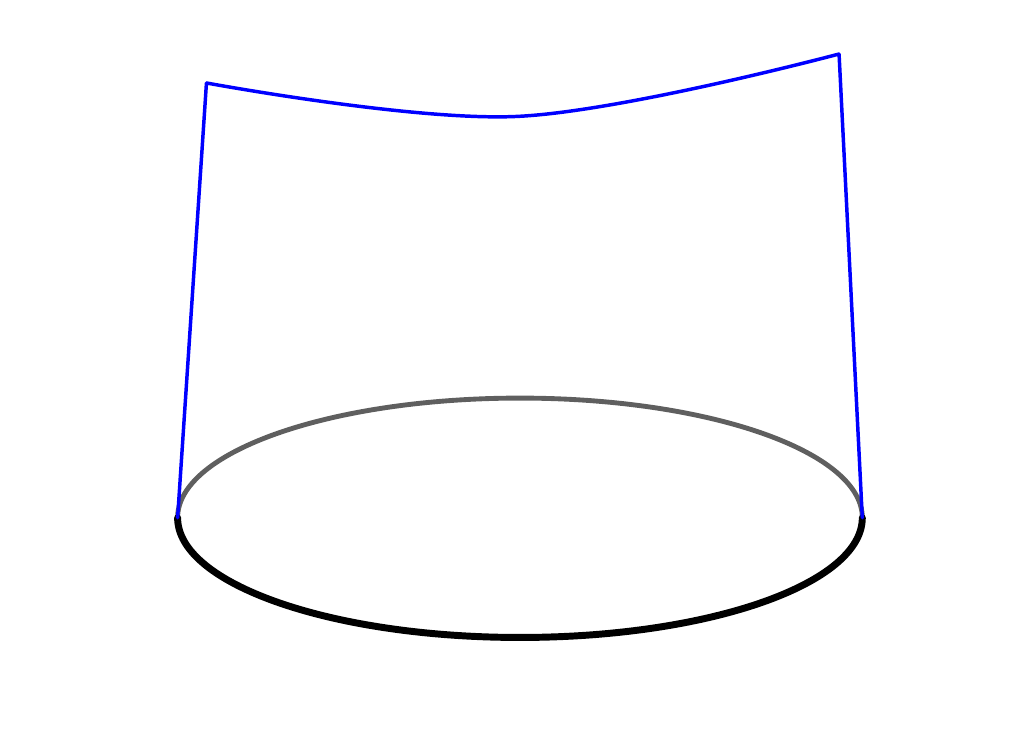}
  \caption{Ribbon tile with stabiliser \orbcap{22\infty}.  }\label{subfig:22star_rc_orb}
  \end{subfigure}
  \\
   \centerline{
\begin{subfigure}[t]{0.3\textwidth}
  \centering
    \includegraphics[width=\textwidth]{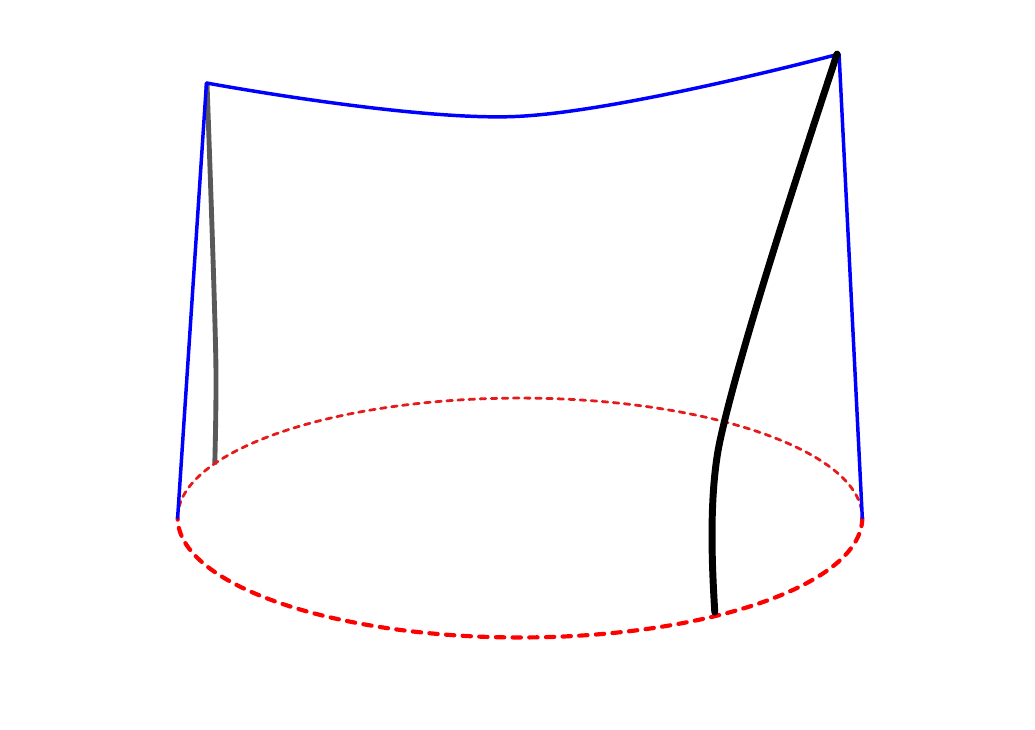}
    \caption{ Ribbon tile with stabiliser \orbcap{\star\infty\infty}. }\label{subfig:22star_rd_orb}
  \end{subfigure}
 \hspace{4em}
  \begin{subfigure}[t]{0.3\textwidth}
  \centering
    \includegraphics[width=\textwidth]{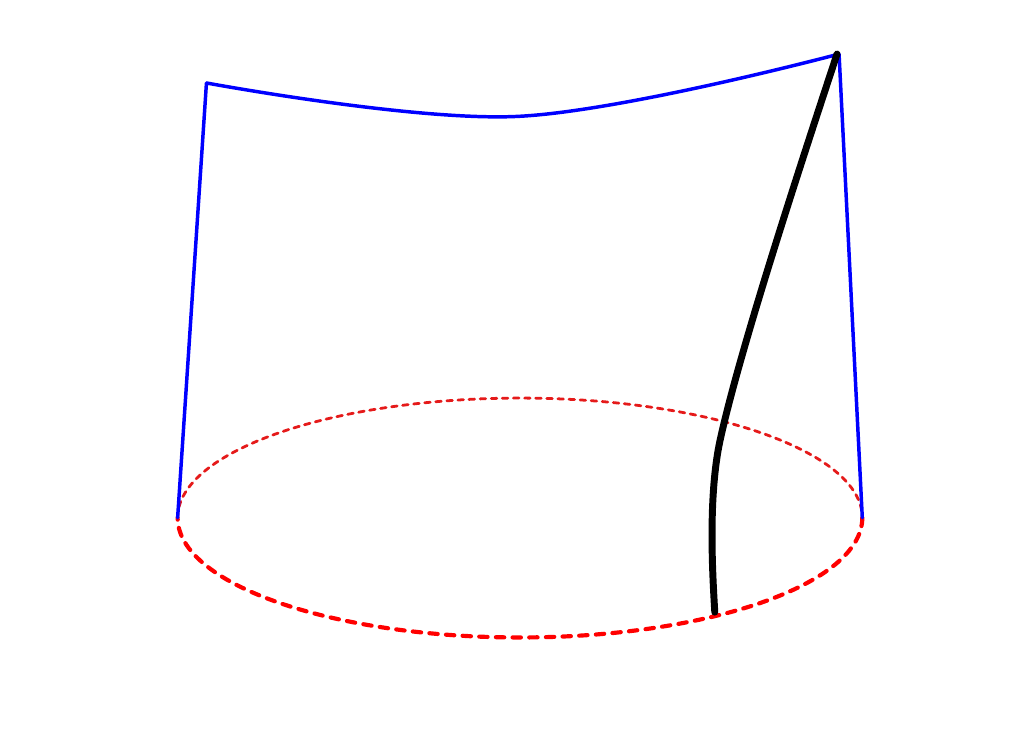}
  \caption{ Ribbon tile with stabiliser \orbcap{2\star\infty}. }\label{subfig:22star_re_orb}
  \end{subfigure}
}
 \caption{By deleting subsets of edges from the four fundamental domain edge graphs shown in figures \ref{fig:EuclideanEx} and  \ref{fig:EuclideanEx2}, we obtain five possible non-fundamental tilings of \orbcap{22\star} with frieze group stabilisers.}  
 \label{fig:EuclideanEx3}
\end{figure}

\begin{figure}
 \begin{subfigure}[t]{0.3\textwidth}
  \centering
    \includegraphics[width=\textwidth]{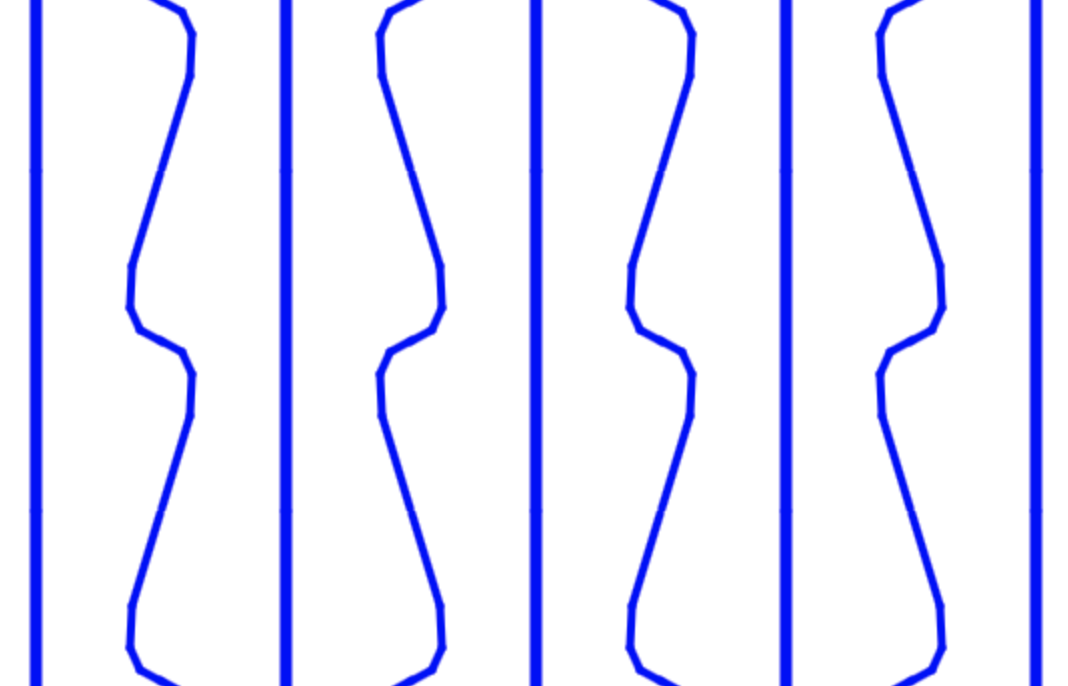}
    \caption{Ribbon tile with stabiliser \orbcap{\infty\infty}.}\label{subfig:22star_ra}
  \end{subfigure}
  \hfill
  \begin{subfigure}[t]{0.3\textwidth}
  \centering
    \includegraphics[width=\textwidth]{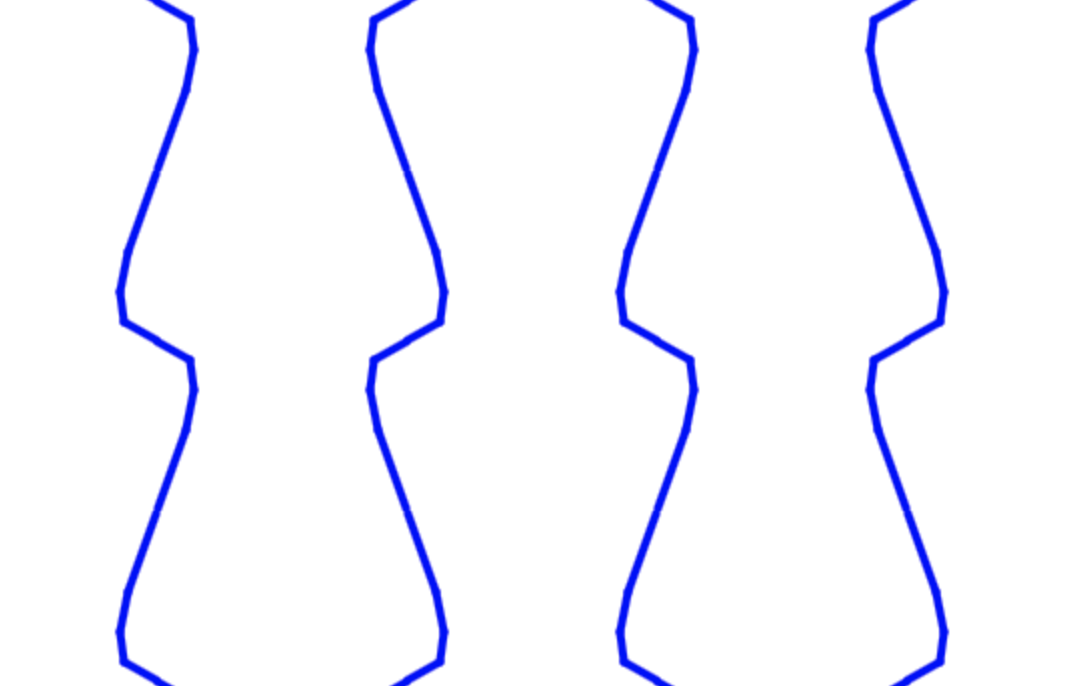}
  \caption{ Ribbon tile with stabiliser \orbcap{\infty\star}. }\label{subfig:22star_rb}
  \end{subfigure}
    \hfill
  \begin{subfigure}[t]{0.3\textwidth}
  \centering
    \includegraphics[width=\textwidth]{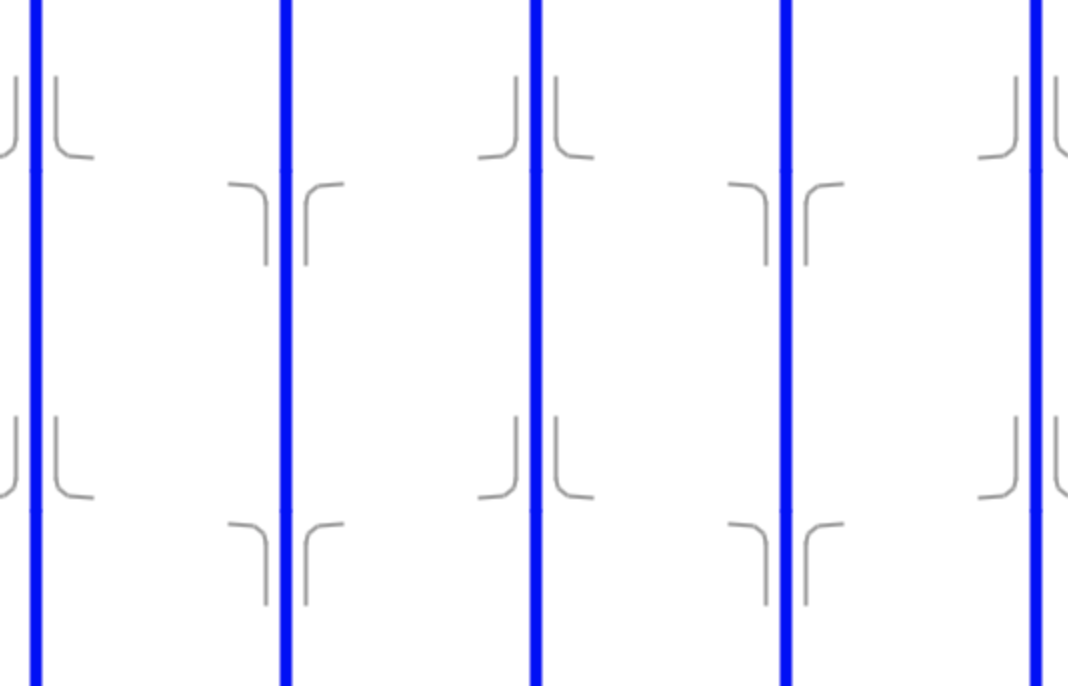}
  \caption{Ribbon tile with stabiliser \orbcap{22\infty}.  }\label{subfig:22star_rc}
  \end{subfigure}
  \\
\centerline{    
    \begin{subfigure}[t]{0.3\textwidth}
  \centering
    \includegraphics[width=\textwidth]{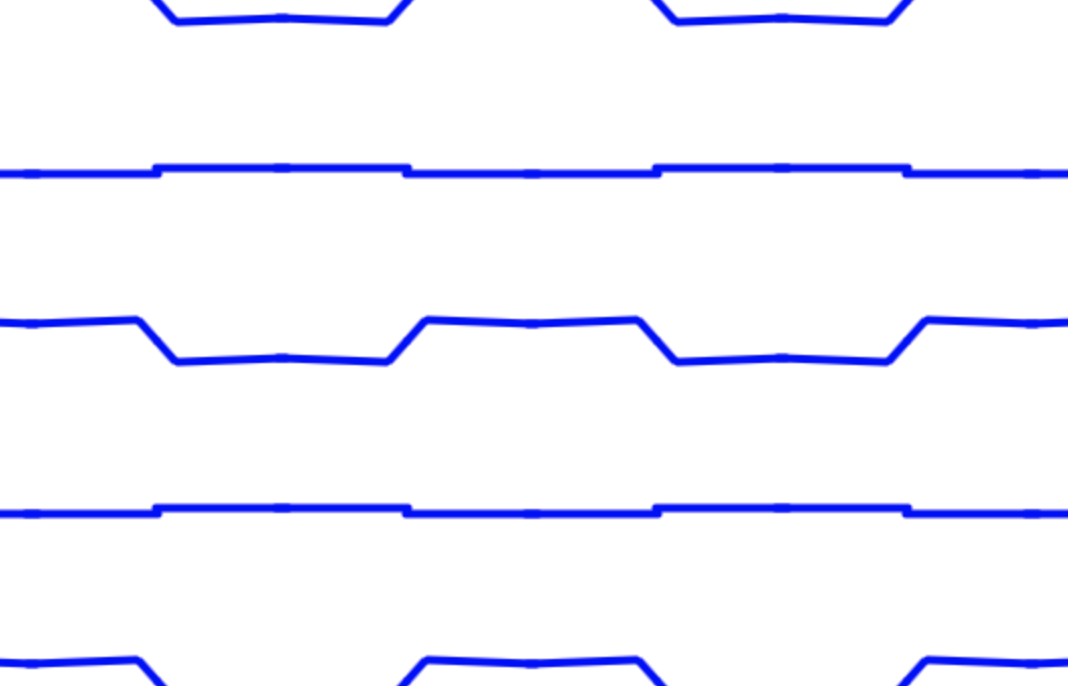}
    \caption{Ribbon tile with stabiliser \orbcap{\star\infty\infty}. }\label{subfig:22star_rd}
  \end{subfigure}
 \hspace{4em}
  \begin{subfigure}[t]{0.3\textwidth}
  \centering
    \includegraphics[width=\textwidth]{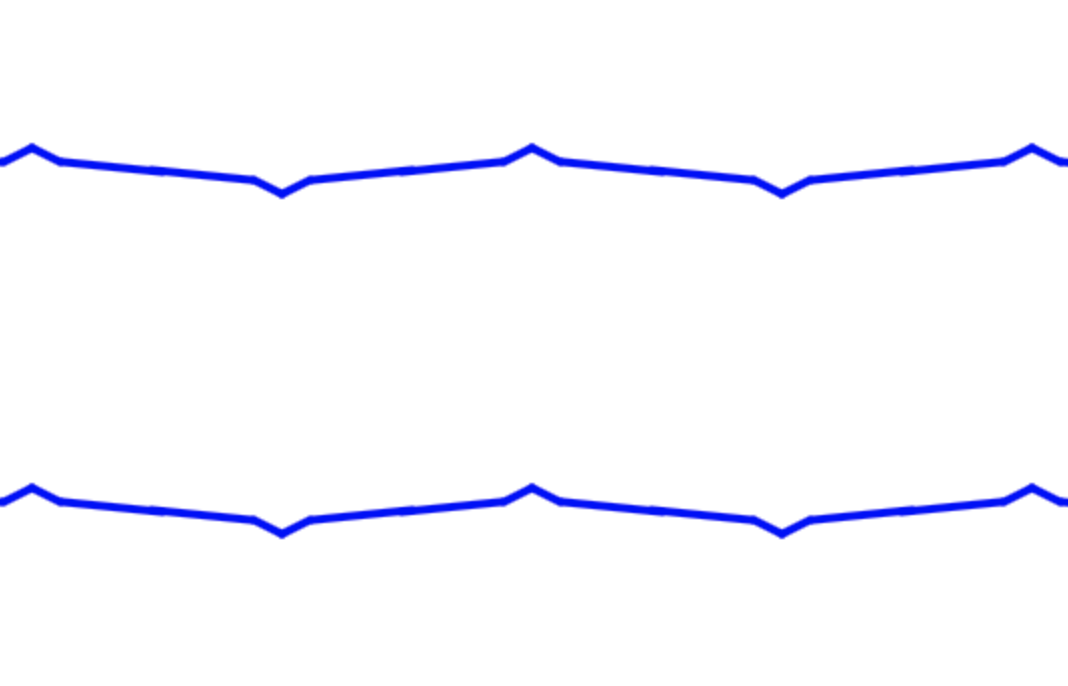}
  \caption{ Ribbon tile with stabiliser \orbcap{2\star\infty}. }\label{subfig:22star_re}
  \end{subfigure}
  }
   \caption{These ribbon tiling patterns correspond directly to the quotient graph diagrams in figure \ref{fig:EuclideanEx3}. Each drawing shows 2 by 2 translational unit cells in the symmetry group \orbcap{22\star}.  The symmetry of the ribbon in (c) can only be geometrically depicted by marking the tiles, e.g. with an `L' motif. }\label{fig:EuclideanEx4}
\end{figure}

The resulting list of tilings from this enumeration will naturally contain equivariantly equivalent duplicates. 
For the tile-1-transitive cases where the orbifold topology and singular locus is not too complex, it is possible to determine the equivalence classes by eye.  
For the more complex orbifolds, and tile-k-transitive cases, it is desirable to have a computable invariant. 
As already known, the D-symbol provides such an invariant for tilings by disks. 
We now present a method for the general case, based on introducing coloured edges to the tiles with infinite stabiliser group to obtain a classical tiling-by-disks. 
As there is more than one way to introduce the coloured edges, we describe how to find a minimal symbol, and call this the \emph{coloured D-symbol}.

Given a ribbon tile-$k$-transitive tiling of the Euclidean or Hyperbolic plane, first identify a crystallographic symmetry group acting on the tiling $(\mathcal{T},\Gamma)$ and some fundamental domain $F$ for the symmetry group.   
Map this fundamental domain and the tile pattern onto the symmetry group quotient space, i.e. its orbifold, $\mathcal{O}$.  
Note that the pattern of tile ``edges'' on the orbifold is not necessarily connected and may contain simple closed curves with no natural ``vertex'' (i.e. no intersection with the singular locus of $\mathcal{O}$).  
We therefore place a single vertex at an arbitrary point on any such loop.  
The next step is to cut up the orbifold along the existing edges (coloured black) to obtain $k$ pieces, with tile vertices, orbifold cone points, mirror boundaries and corner points marked with their order.  
Any piece that comes from a bounded tile can be treated as per the usual D-symbol process of barycentric subdivision.  
All other pieces require extra cuts to create a polygonal domain for barycentric subdivision.  We will colour these edges green. 
\begin{enumerate}
\item Let $T$ be a tile from $\mathcal{T}$ with infinite stabiliser group $H \subset \Gamma$, and let $T/H$ be the decorated quotient space with black tile edges and vertices, cone points, mirror boundaries and corner points marked.  $T/H$ is a compact 2-manifold with $m$ boundary components, where these can now be mirror boundaries or tile edges. Cap the boundary components to obtain a closed surface $\mathcal{A}$ and compute the genus, $g(\mathcal{A}) \geq 0$ and orientability $\alpha$. 
\item Add green edges, $M$, to all mirror boundaries not already coloured black.  If the mirror boundary has no corner points or intersection with black tile edges, then insert a vertex at an arbitrary point.  
\item If $g(\mathcal{A}) > 0$, construct all possible graphs $C$ consisting of a single vertex $v$ and $\alpha g$ loops (with $\alpha = 2$ if $\mathcal{A}$ is orientable and $\alpha = 1$ if not) such that $\mathcal{A} \setminus C$ is a topological disk. We then embed the graph $C$ in $T/H$, treated as a subspace of $\mathcal{A}$.
\item Discard a neighbourhood of the vertex $v$ to leave $\alpha g$ disjoint paths with $2\alpha g$ endpoints $v_i$.  
\item Construct all possible trees $S$ on $T/H$ that do not intersect the $\alpha g$ paths (except at the endpoints) such that $S$ spans the following vertices: 
\begin{enumerate}
 \item the path endpoints $v_i$, each of degree one. 
 \item cone points not already marked as belonging to a tile edge. 
 \item a single point on each boundary component, at an existing vertex of minimal degree.
\end{enumerate}
\item Let $U$ be the union of the green edges from $M$, one choice of $C$, and one choice of $S$. 
\item (*) From all possibilities for $U$, keep those such that $T/H \setminus U$ is a combinatorial polygon with a minimal number of edges. 
\end{enumerate}
The above algorithm is adapted from~\cite{Lucic2018}, but we have restricted it to insert as few vertices and edges as possible.  
Note that there must be at least one tile boundary component in $T/H$, so the construction of the tree, $S$, is well defined.  
Illustrations are provided for an example in the following section. 

Once we have candidate sets of green edges $U$, that cut each tile class into a minimal polygon, we reassemble the fragments to form a tiling-by-disks with two types of edge colour. 
From this, we create a D-symbol for each possible combination of cuts found in step $(*)$.  
A canonical representative D-symbol for the ribbon tiling will be the one that appears first in the ordering as defined in~\cite{Delgado-Friedrichstilings}. 
Two such D-symbols will encode equivariantly equivalent tilings if they are isomorphic in the usual sense, with the extra requirement that the isomorphism preserves the edge-colours. 

The proof that this algorithm leads to a unique D-symbol encoding a ribbon tiling follows from the corresponding statement for classical tilings, for which the isomorphism class of a D-symbol completely characterises the isomorphism class of the symmetry group and the tiling combinatorics. 
The crux of the matter is that the above recipe leads to a unique classical tiling and D-symbol starting from any geometric realisation of the ribbon tiling. 
The constructions involved are made in the quotient space $T/H \subset \mathcal{O}$ and so do not depend on the particular realization of the tiling in $\mathcal{X}$.  
Also, the edges we insert must be part of a graph that determines a fundamental tiling of $T/H$.  
The uniqueness is guaranteed by the fact that there are a finite number of combinatorially distinct fundamental tilings for a given wallpaper or NEC group, and the possibility to rank D-symbols by the methods of~\cite{Delgado-Friedrichstilings}.

\bigskip

The enumeration of non-fundamental tilings, in both the classical (disk-like) and ribbon tiling cases effectively constructs all the infinite-index subgroups generated by a subset of the parent group generators.  
The subgroups are the stabiliser groups of the non-fundamental tile while the parent group generators are dictated by the particular fundamental domain that we start with. 
We observe that some stabiliser subgroups are derivable from all fundamental domains for a particular group, but this does not always hold.  
In fact, there are some non-fundamental tilings that arise from edge deletion in just one fundamental domain (see figure~\ref{subfig:22star_rd_orb} for example).  

\section{Hyperbolic tiling examples}\label{sec:ex}

In this section we look at further examples of ribbon tilings that illustrate the theory developed above and show how to apply our results in practice. We start with an example of the classification of a ribbon tiling via coloured D-symbols.

Consider the ribbon tiling in figure \ref{fig:2223_1}. We identify its symmetry group as $G=\orb{2223}$. Figure \ref{subfig:2223_1b} shows a fundamental domain with the cone points on its boundary. We see that this ribbon tiling has only one class of edge and tile. 
The orbifold is (topologically) a sphere with four cone points, and the ribbon edge joins the cone point of order 3 to one of order 2; see figure \ref{subfig:2223_1c}. 
There are two possible trees that span the ribbon vertex and the two other cone points, as required by the algorithm at the end of Section \ref{sec:gendsymbs}, one of these is drawn with green edges in figure \ref{subfig:2223_1c}. 
The corresponding tilings are shown in figures \ref{subfig:2223_2a} and \ref{subfig:2223_2b}. 
The second one has only three edges bordering each fundamental tile, hence it is the simplest classical tiling encoding this ribbon tiling. 
Barycentric subdivision as shown in figure \ref{subfig:2223_2c} leads to the coloured D-symbol shown in figure~\ref{fig:colouredDsymbol}.

\begin{figure}[!tbp]
  \begin{subfigure}[t]{0.3\textwidth}
  \centering
    \includegraphics[width=\textwidth]{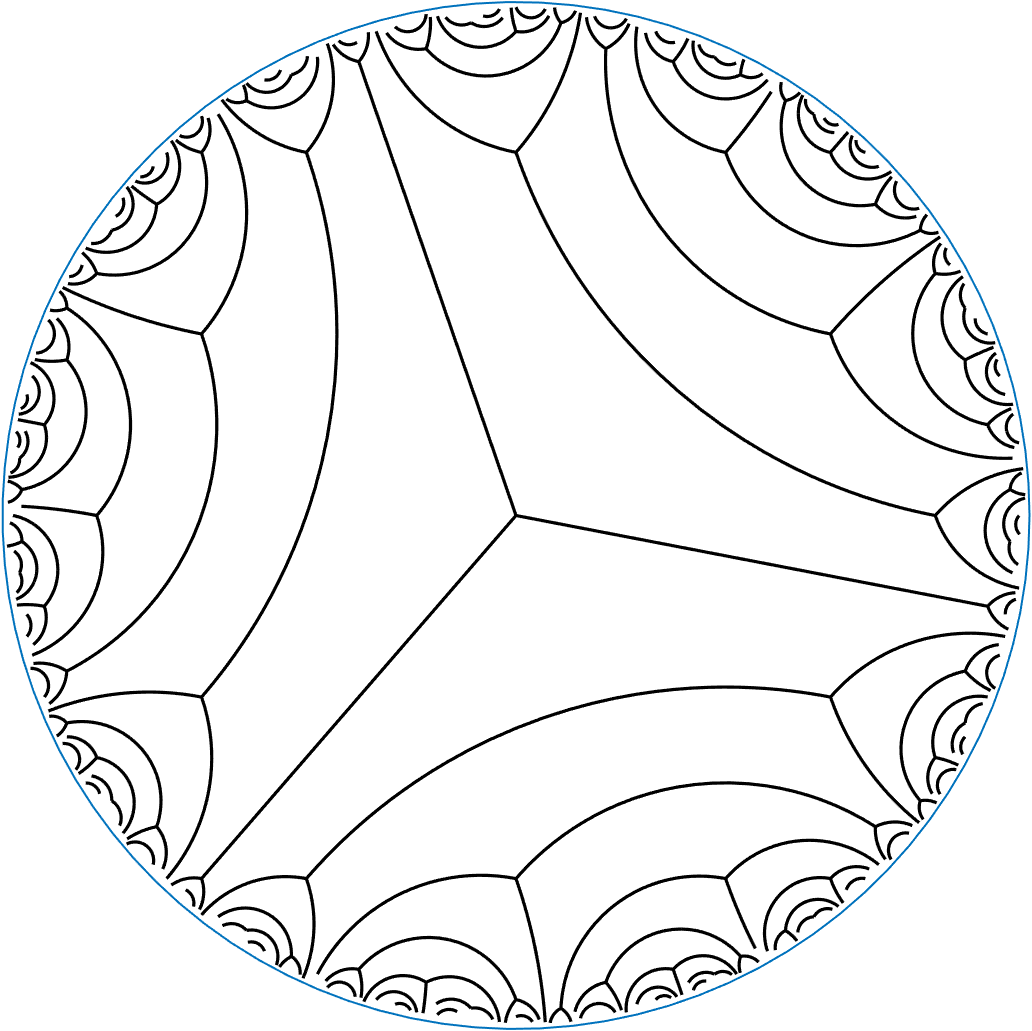}
    \caption{A tile-$1$-transitive ribbon tiling with symmetry group \orb{2223}.}\label{subfig:2223_1a}
  \end{subfigure}
  \hfill
  \begin{subfigure}[t]{0.3\textwidth}
  \centering
    \includegraphics[width=\textwidth]{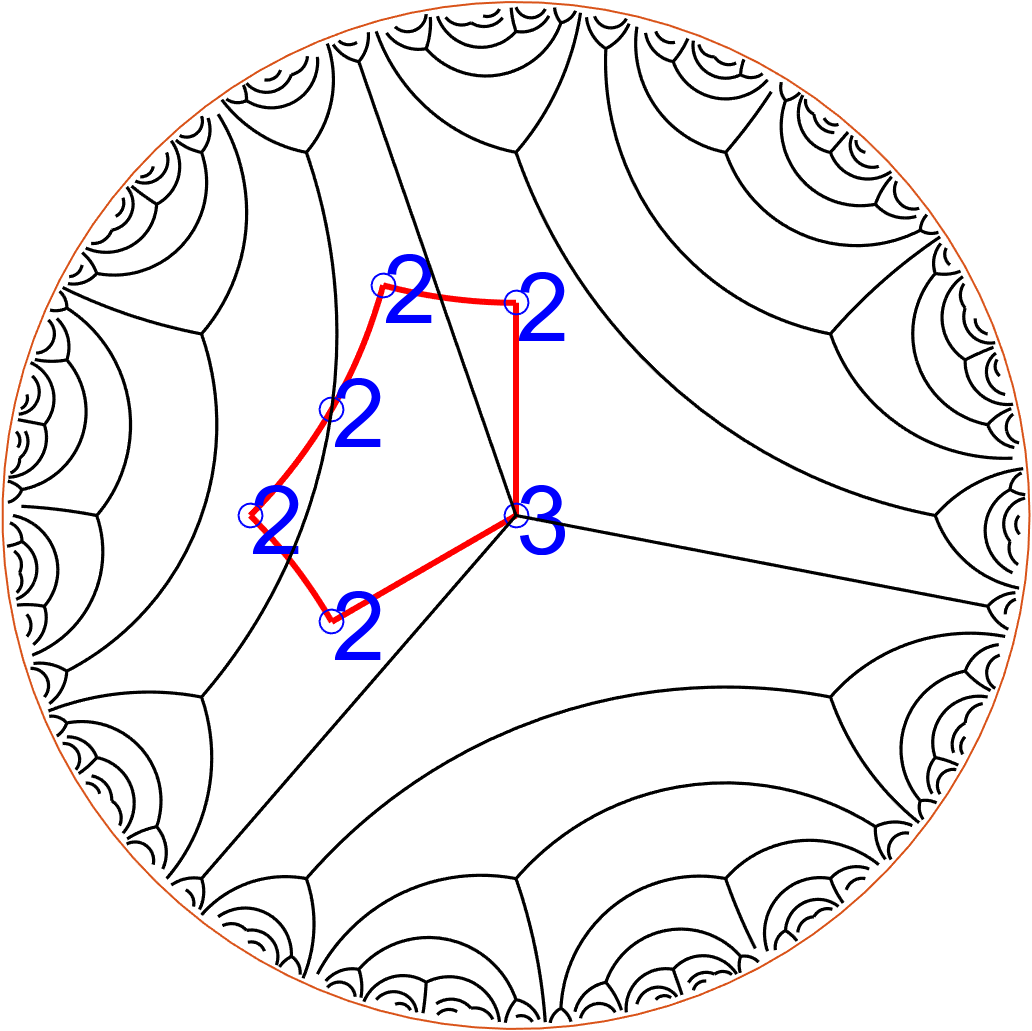}
  \caption{The ribbon tiling from (a) with a fundamental domain FD of the symmetry group. The points of increased symmetry on the boundary of FD are marked with their orders.}\label{subfig:2223_1b}
  \end{subfigure}
    \hfill
  \begin{subfigure}[t]{0.3\textwidth}
  \centering
    \includegraphics[width=\textwidth]{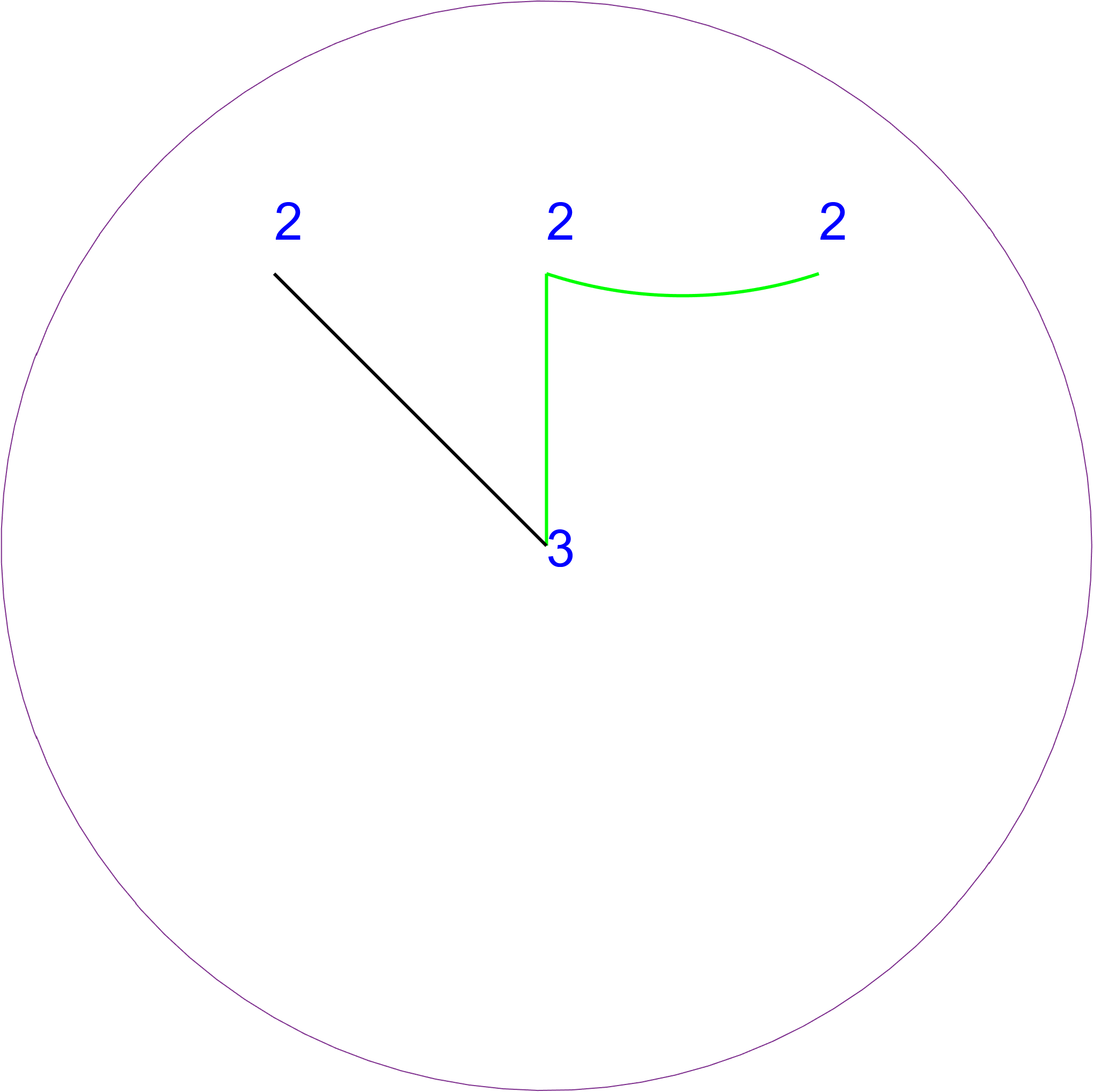}
  \caption{The orbifold is topologically a sphere with $4$ marked points. The green edges are introduced to form a tree.}\label{subfig:2223_1c}
  \end{subfigure}
\\
   \begin{subfigure}[t]{0.3\textwidth}
  \centering
    \includegraphics[width=\textwidth]{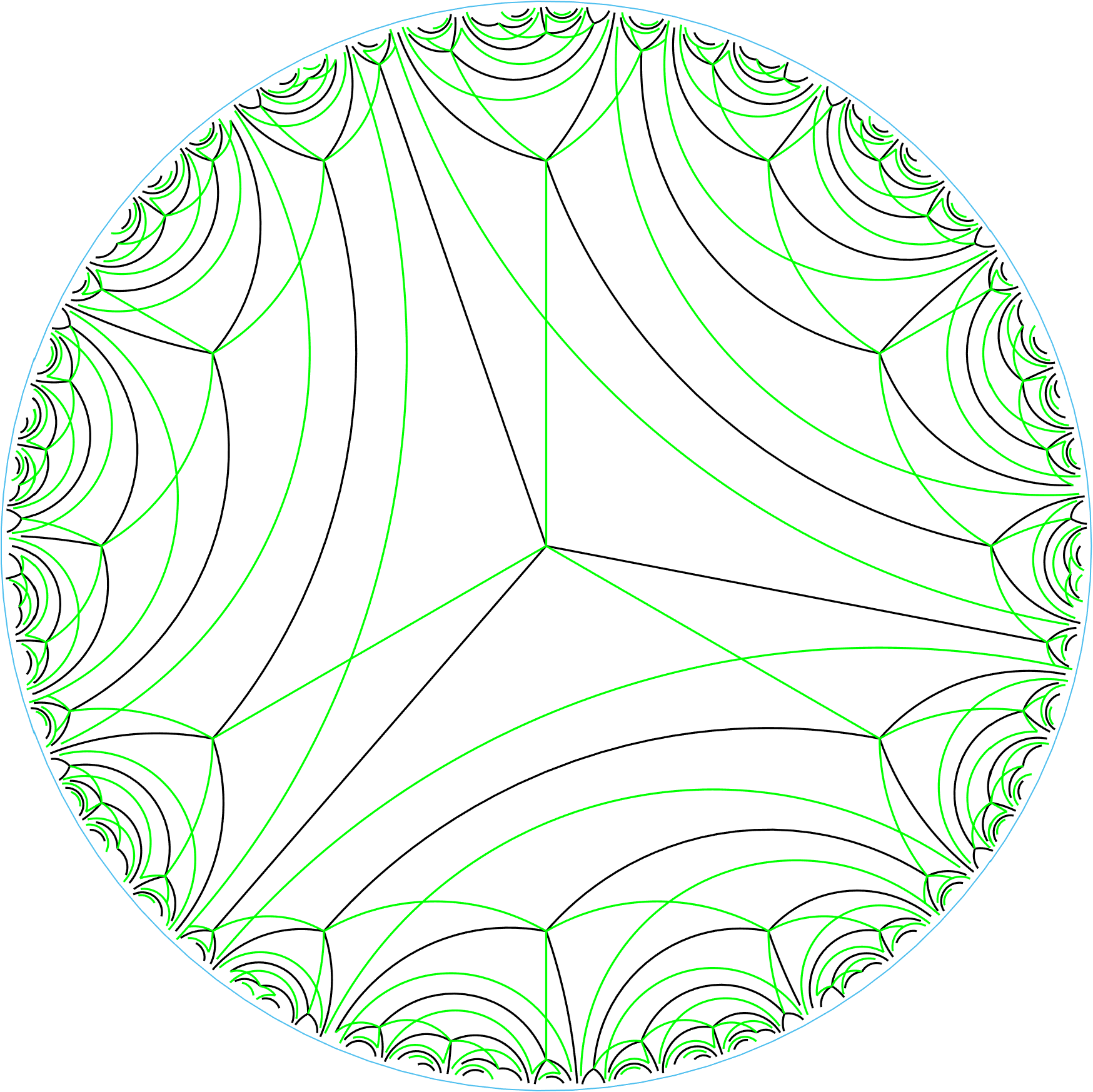}
  \caption{A classical tiling resulting from the insertion of the edges indicated in figure \ref{subfig:2223_1c}. Each fundamental tile is bordered by four edges.}\label{subfig:2223_2a}
  \end{subfigure}
  \hfill
  \begin{subfigure}[t]{0.3\textwidth}
  \centering
    \includegraphics[width=\textwidth]{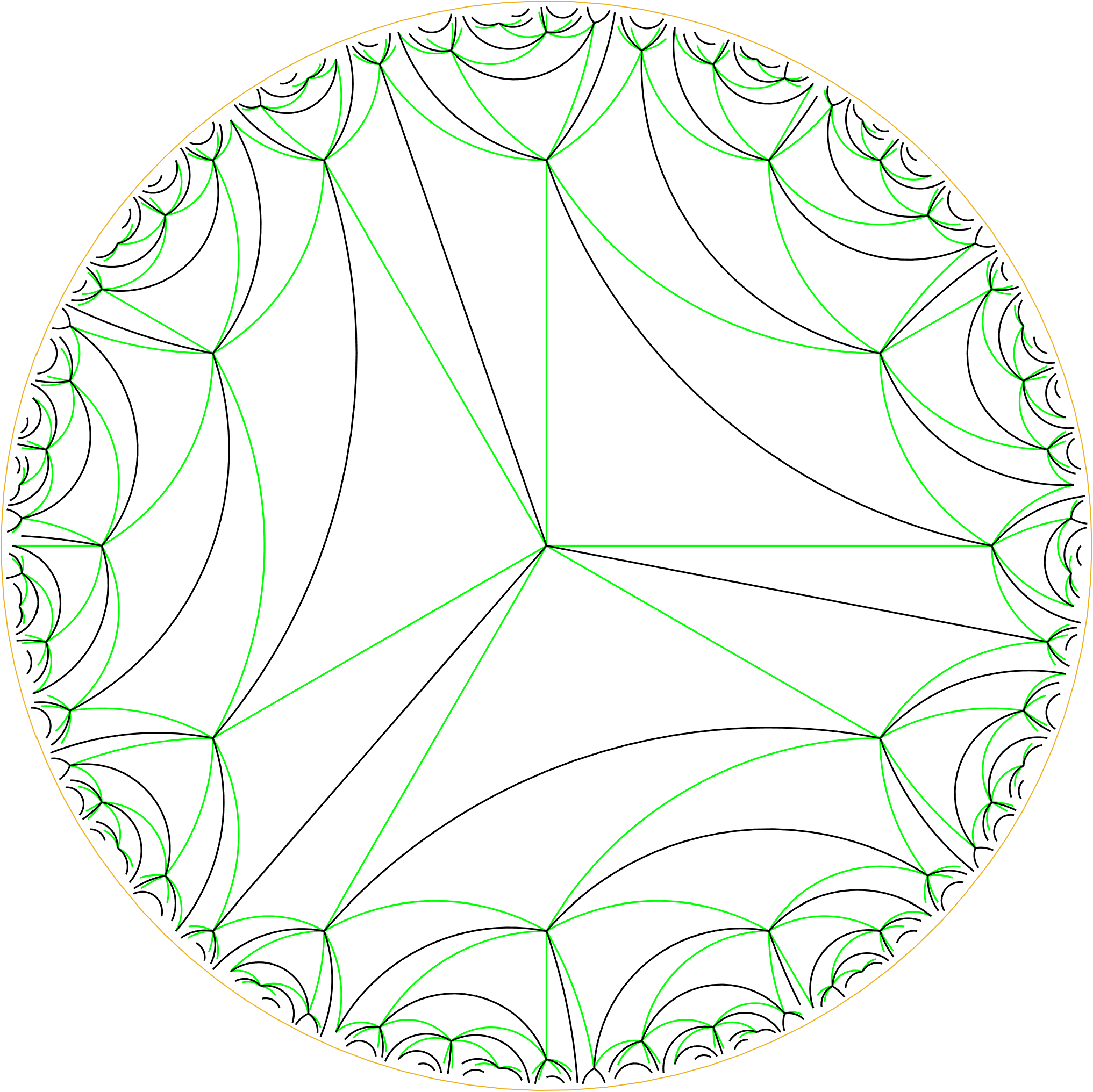}
  \caption{The simplest classical tiling, with fundamental tiles bordered by only three edges.}\label{subfig:2223_2b}
  \end{subfigure}
\hfill
  \begin{subfigure}[t]{0.3\textwidth}
  \centering
    \includegraphics[width=\textwidth]{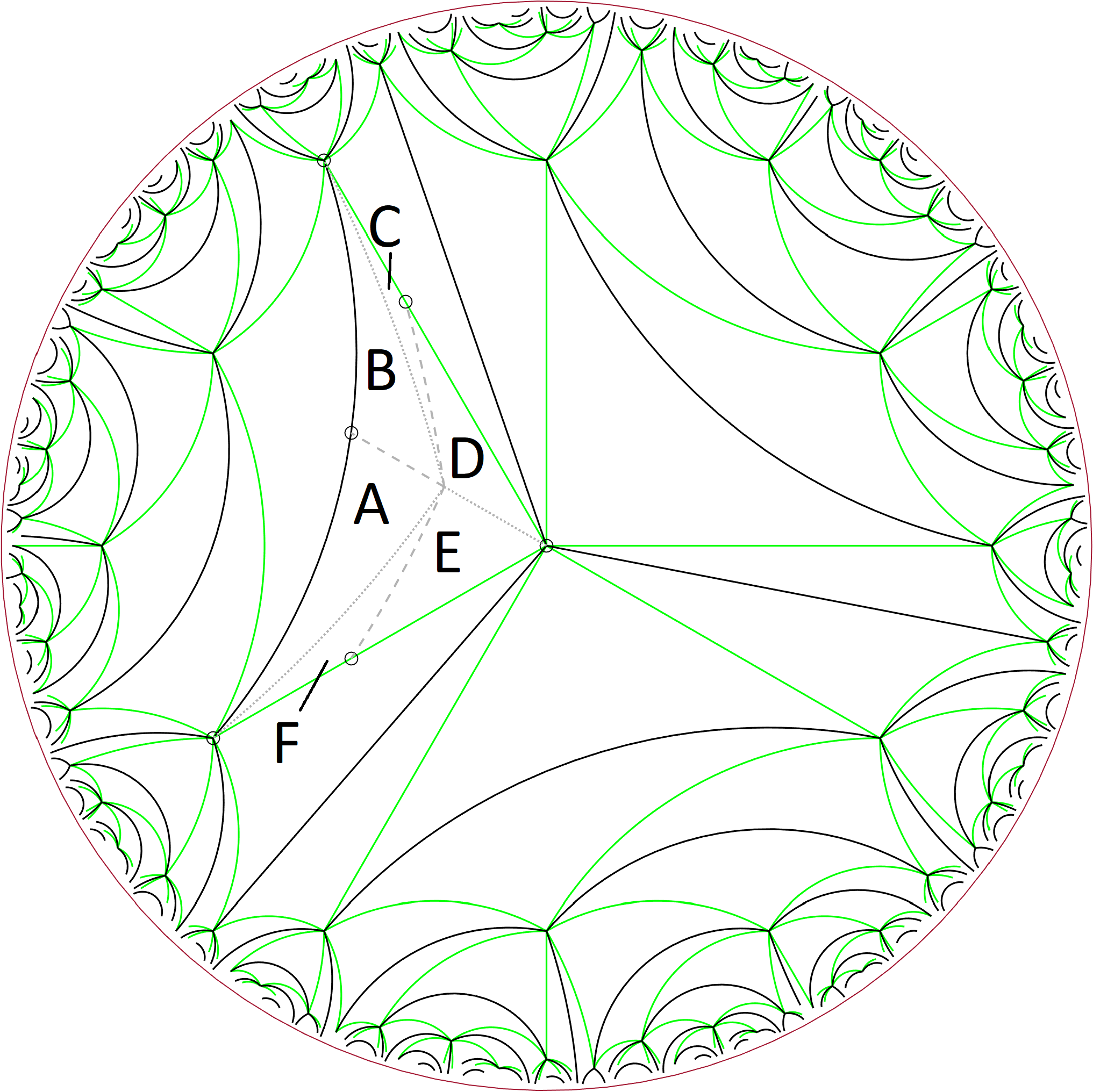}
  \caption{Barycentric subdivison of the simplest classical tiling leads to the coloured D-symbol for the ribbon tiling.}\label{subfig:2223_2c}
  \end{subfigure}
    \caption{Ribbon tiling with symmetry group \orb{2223}.}\label{fig:2223_1}
\end{figure}

\begin{figure}[!tbp]
 \centering
 \includegraphics[width=0.7 \textwidth]{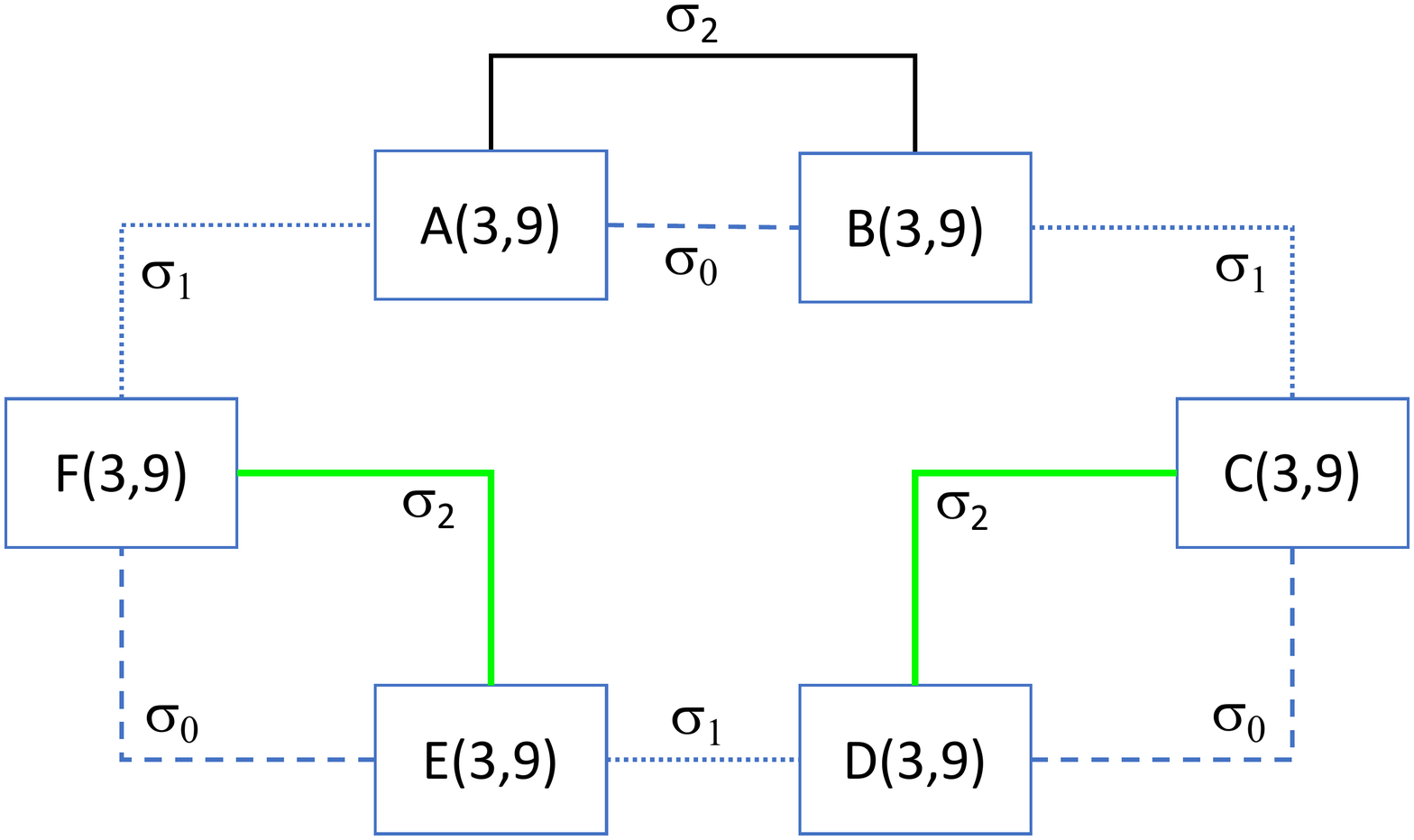}
 \caption{The coloured D-symbol for the ribbon tiling from figure \ref{fig:2223_1} with symmetry group \orb{2223}.
As in the classical setting, the D-symbol is a graph that records adjacencies between triangles (or chambers) using different line styles for adjacencies opposite a vertex (dashed), edge (dotted) and face (solid).  The two numbers associated with each chamber are face and vertex indices: $(f,v)$ where $(\sigma_0 \sigma_1)^{f} (A) = A$ and $(\sigma_1 \sigma_2)^v (A) = A$ are orbits in the plane $\mathcal{X}$ that return to the starting chamber, $A$.  Note that it is always the case that $(\sigma_0 \sigma_2)^2(A) = A$ as there are four chambers around the midpoint of an edge.  See~\cite{DRESS1987,Dresshu,Delgado-Friedrichstilings} for further details of standard D-symbols.  Our coloured D-symbol augments this graph by recording the black or green colour for the tile edges shared by $\sigma_2$-adjacent chambers. }\label{fig:colouredDsymbol}
\end{figure}

Now we will show how to enumerate the ribbon tilings with symmetry group \orb{2224}. 
There are nine combinatorially distinct fundamental domain tilings with this symmetry, shown in figure \ref{fig:2224funds}. More information about these tilings, coloured pictures, their D-symbols and their transitivity classes can be found in \cite{epinetsearch}.
\begin{figure}[!tbp]
  \begin{subfigure}[t]{0.3\textwidth}
  \centering
    \includegraphics[width=\textwidth]{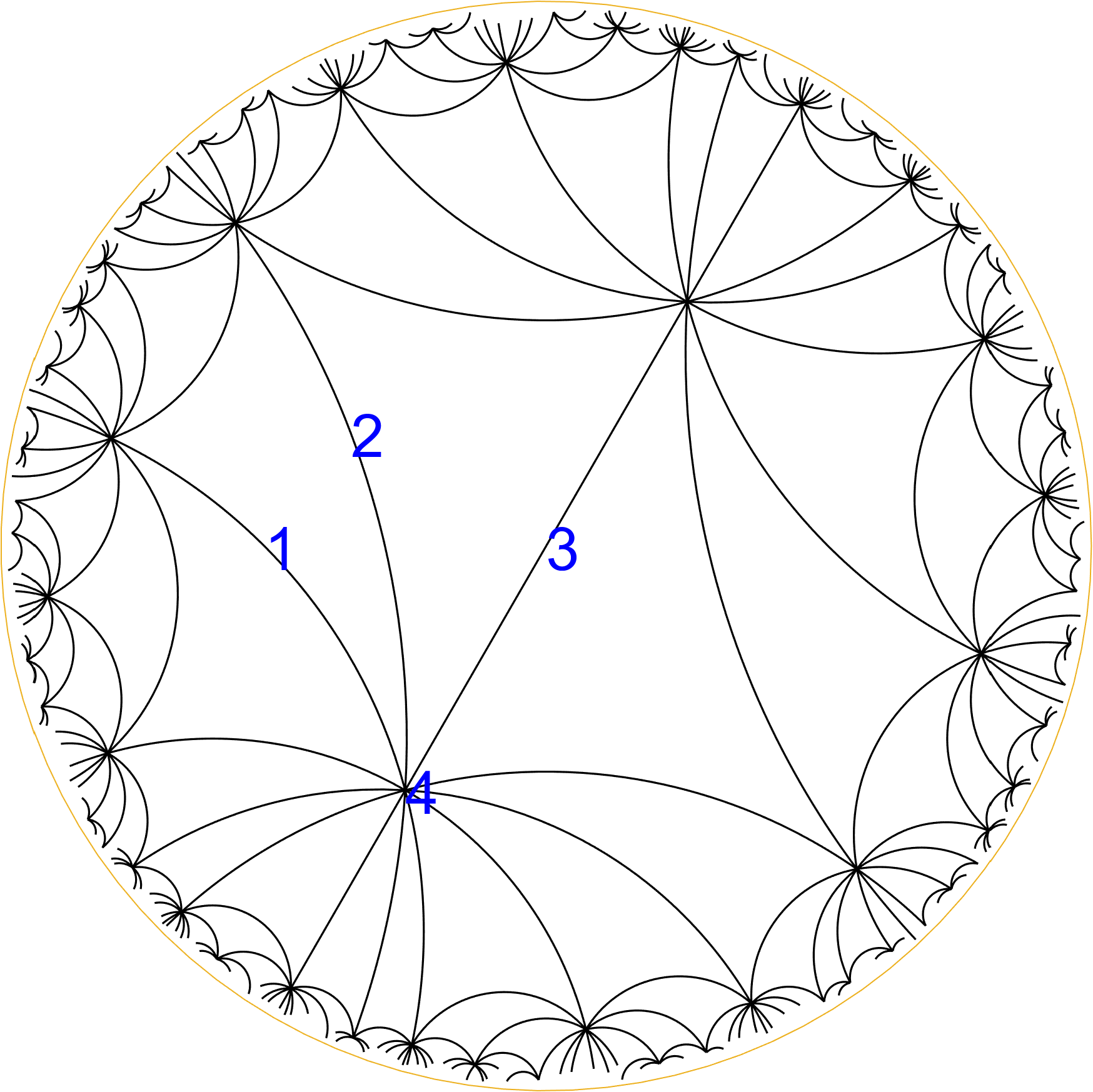}
    \caption{Corresponds to the tiling QS20.}\label{subfig:2224_a}
  \end{subfigure}
  \hfill
  \begin{subfigure}[t]{0.3\textwidth}
  \centering
    \includegraphics[width=\textwidth]{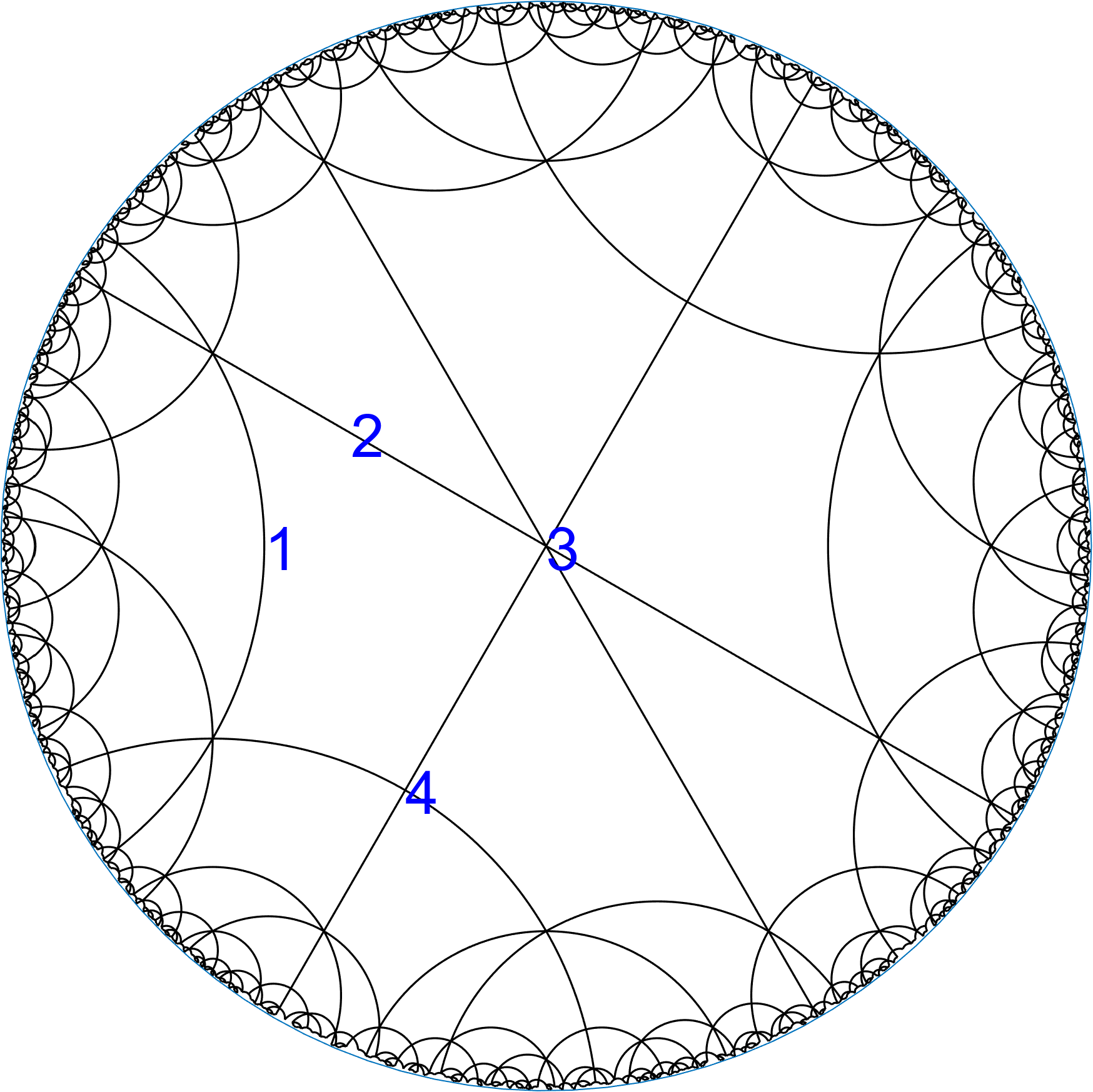}
    \caption{Corresponds to the tiling QS21.}\label{subfig:2224_b}
  \end{subfigure}
  \hfill
    \begin{subfigure}[t]{0.3\textwidth}
  \centering
    \includegraphics[width=\textwidth]{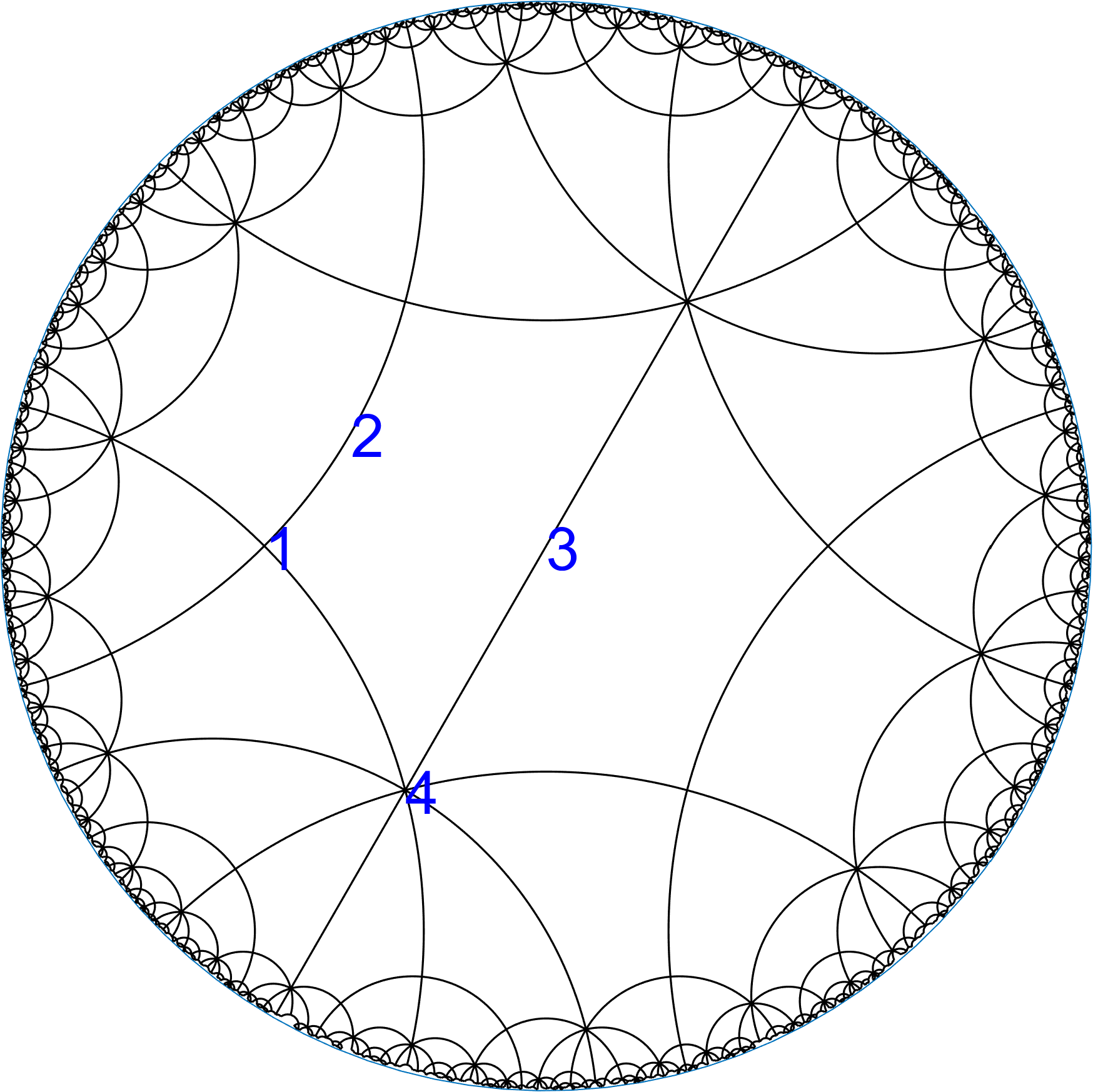}
    \caption{Corresponds to the tiling QS22.}\label{subfig:2224_c}
  \end{subfigure}
  \\
  \begin{subfigure}[t]{0.3\textwidth}
  \centering
    \includegraphics[width=\textwidth]{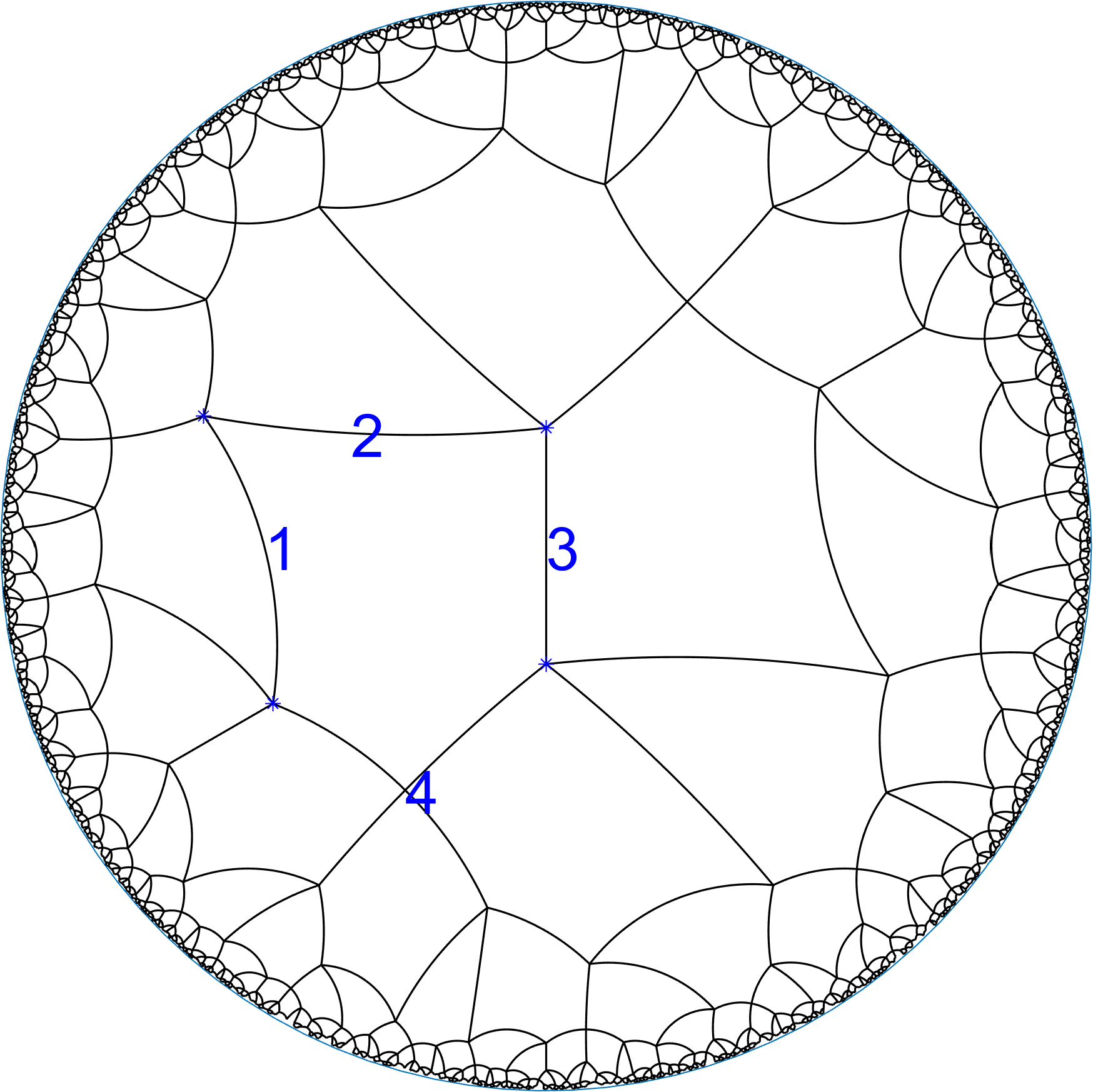}
    \caption{Corresponds to the tiling QS23.}\label{subfig:2224_d}
  \end{subfigure}
  \hfill
  \begin{subfigure}[t]{0.3\textwidth}
  \centering
    \includegraphics[width=\textwidth]{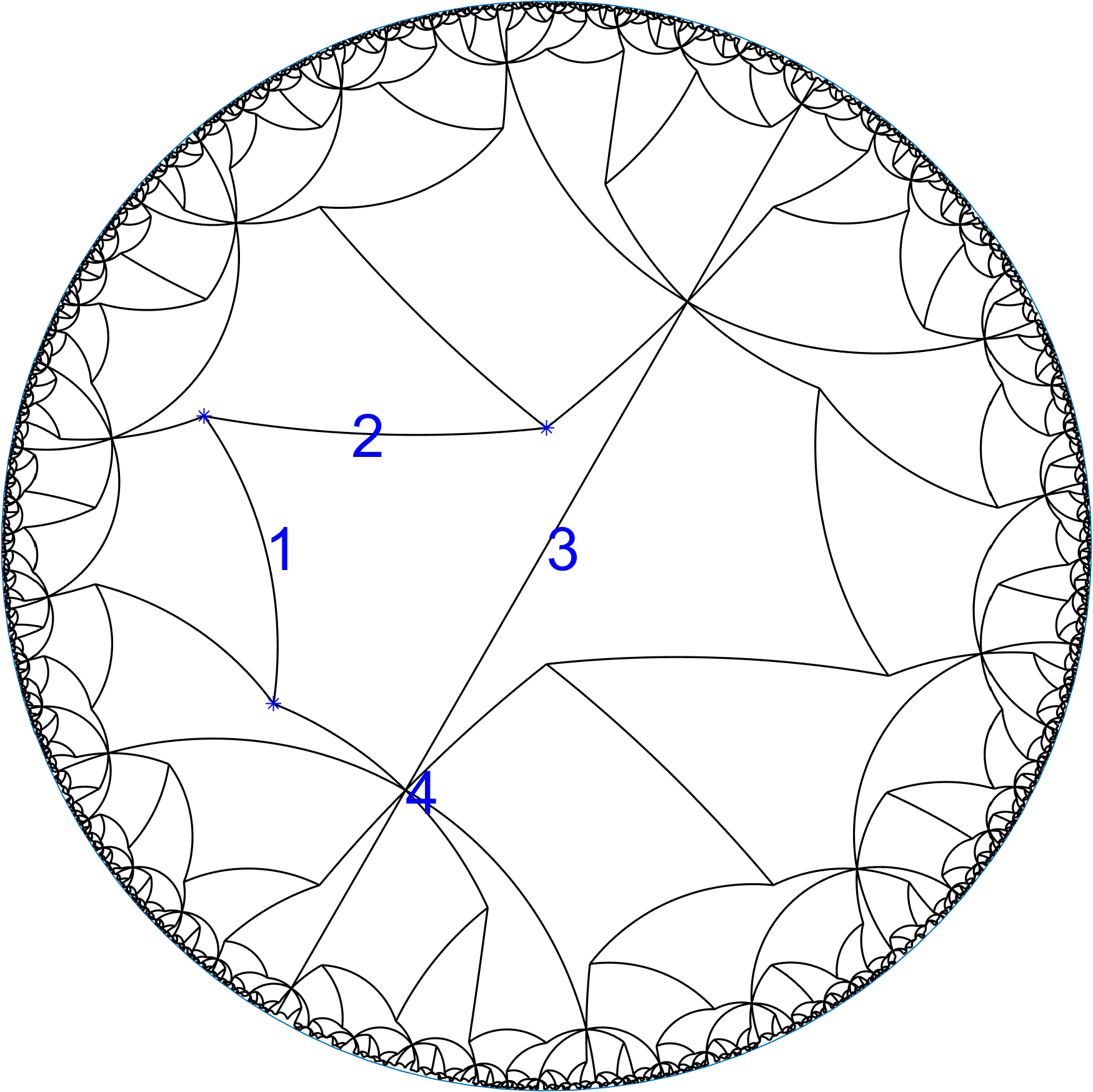}
    \caption{Corresponds to the tiling QS24.}\label{subfig:2224_e}
  \end{subfigure}
  \hfill
  \begin{subfigure}[t]{0.3\textwidth}
  \centering
    \includegraphics[width=\textwidth]{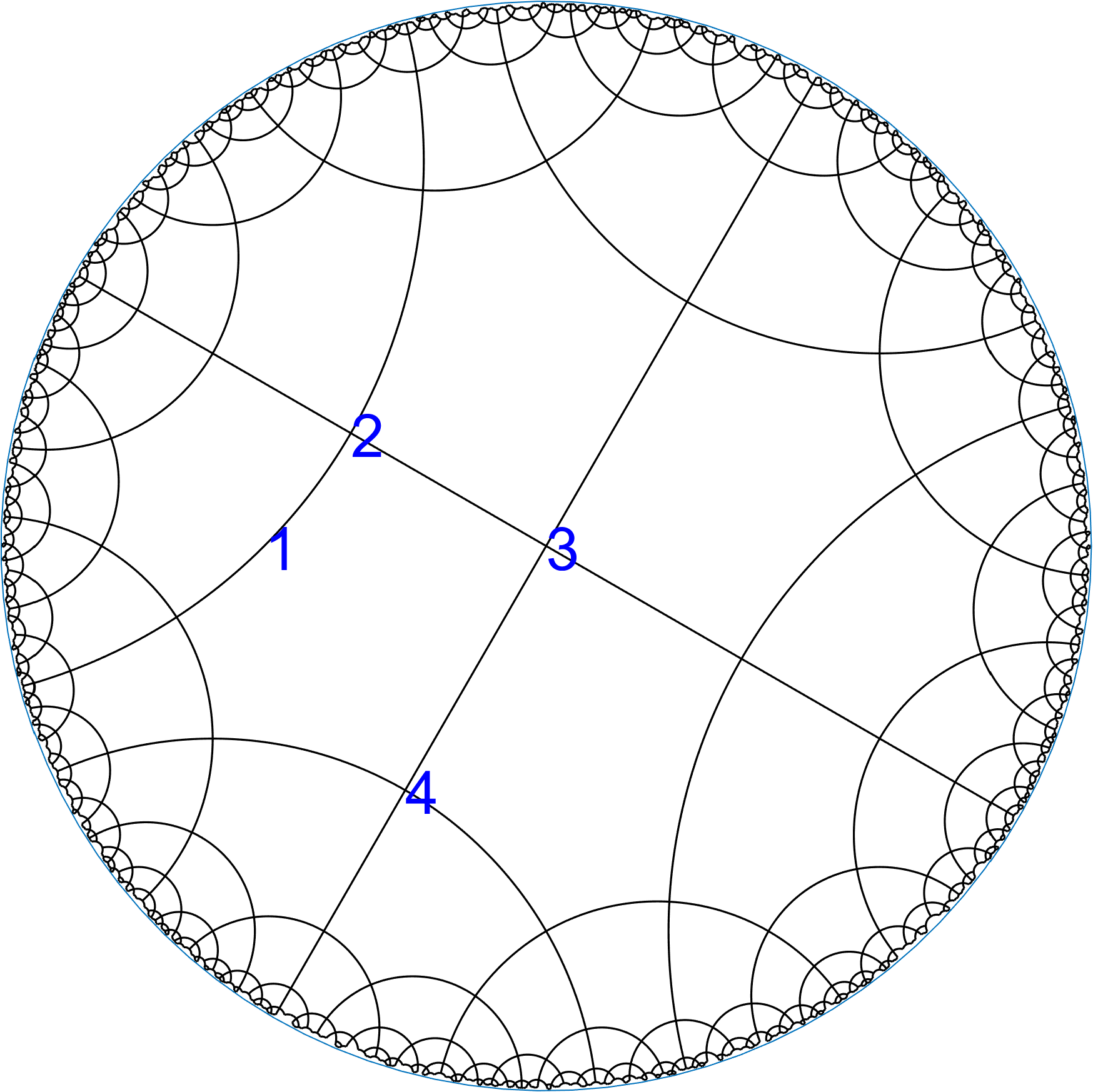}
    \caption{Corresponds to the tiling QS25.}\label{subfig:2224_f}
  \end{subfigure}
  \\
    \begin{subfigure}[t]{0.3\textwidth}
  \centering
    \includegraphics[width=\textwidth]{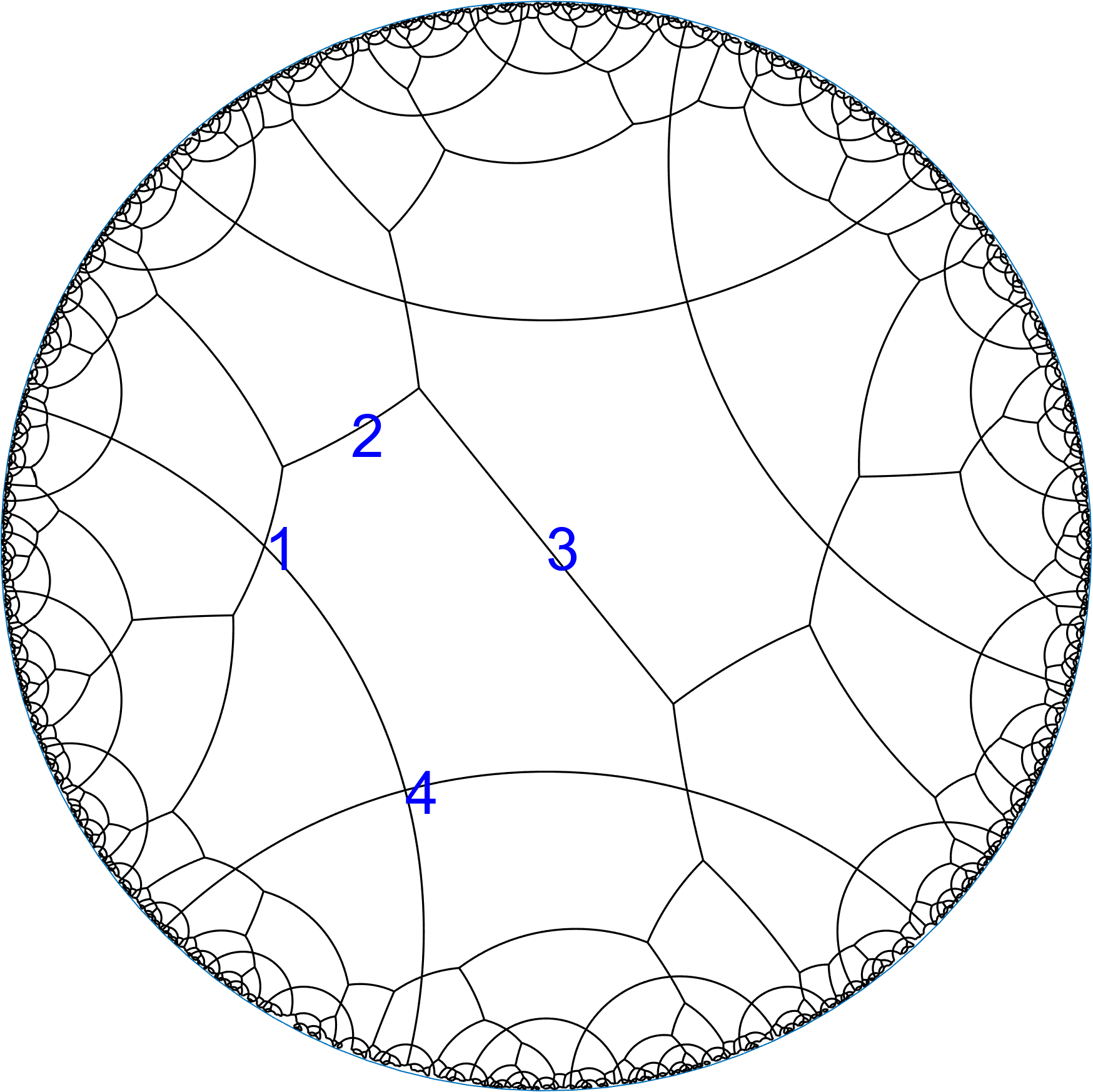}
    \caption{Corresponds to the tiling QS26.}\label{subfig:2224_g}
  \end{subfigure}
  \hfill
  \begin{subfigure}[t]{0.3\textwidth}
  \centering
    \includegraphics[width=\textwidth]{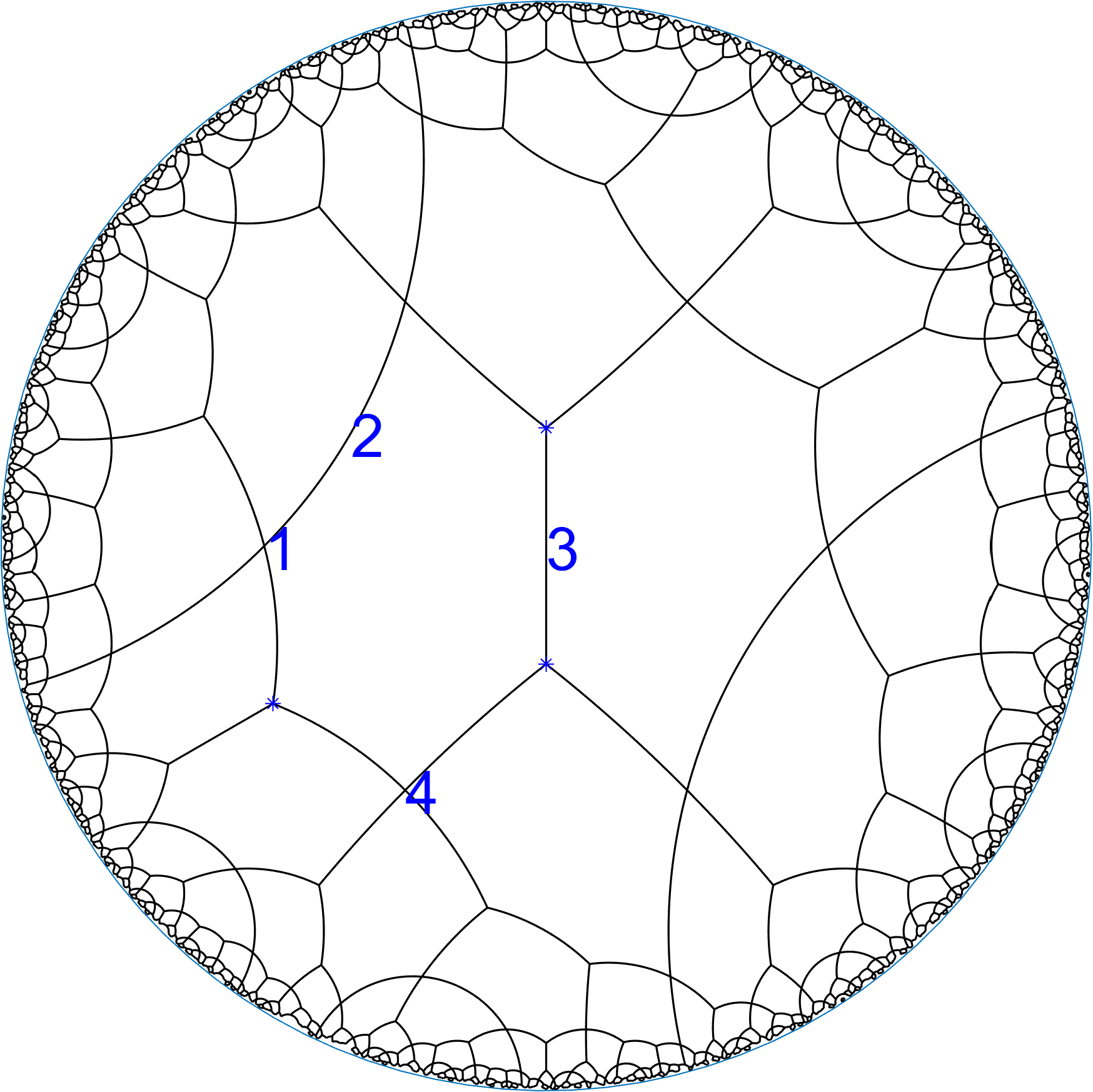}
    \caption{Corresponds to the tiling QS27.}\label{subfig:2224_h}
  \end{subfigure}
  \hfill
  \begin{subfigure}[t]{0.3\textwidth}
  \centering
    \includegraphics[width=\textwidth]{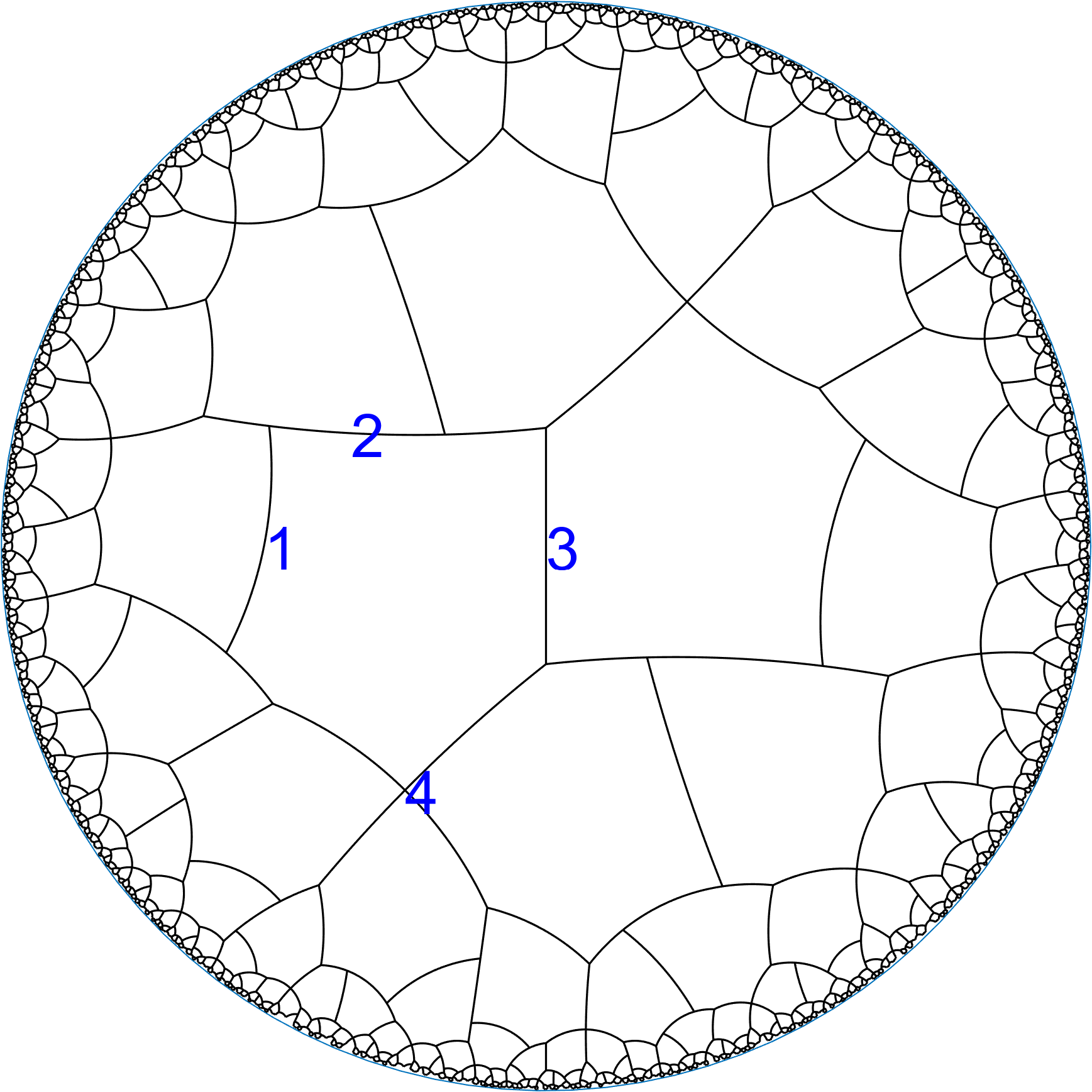}
    \caption{Corresponds to the tiling QS28.}\label{subfig:2224_i}
  \end{subfigure}
    \caption{The different fundamental tilings with symmetry group \orbcap{2224} generated by the indicated generators, with the $4$-fold rotation located at vertex $4$. The naming QS$n$ is that used in the epinet database \cite{epinetsearch}. Further information about the tilings including their D-symbols is accessible at \textsf{epinet.anu.edu.au/QS}$n$. }\label{fig:2224funds}
\end{figure}
Classical non-fundamental tilings are built from these by deleting from the tile-boundary graph $C \subset \mathcal{O}$, a single edge incident at a cone point with degree-1.  This process results in five combinatorially distinct tilings, each with stabiliser group either $\orb{2\bullet} = C_2$ or $\orb{4\bullet} = C_4$.  

When constructing the ribbon tilings, first note the following observations that are a consequence of our results in the previous sections. 
In order to produce a ribbon tile, we need to delete at least two edges from a fundamental domain. 
Furthermore, deleting edges that are non-adjacent in the tile boundary must result in a ribbon tile. 
If the two edges are related by a symmetry, then this symmetry must shift all points in $\mathbb{H}^2$ by some length bounded from below, i.e. the symmetry generates a translation or glide. 
If the two edges are not related by a symmetry then the subgroup generated by the corresponding symmetry operations must be non-abelian, and we are in the situation of Lemma \ref{lemma:infstab}. 

For example, notice that the tiling in figure \ref{subfig:2224_i} can be obtained from the tiling in figure \ref{subfig:2224_d} by splitting the degree-4 vertex between the cone points labelled `2' and `3' into two degree-3 vertices by introducing a new edge.  Observe that the new edge does not intersect any cone points, that it connects two distinct vertices in the orbifold, and is not adjacent to a copy of itself.  Therefore, deleting this new edge must result in a ribbon tile with stabilizer group \orb{\infty\infty}. 
Here, while figure \ref{subfig:2224_i} supports the ribbon tiling in figure \ref{subfig:2224_branchedribbon_a}, the tiling from figure \ref{subfig:2224_d} does not.  

From the nine fundamental domain tilings for \orb{2224}, just three combinatorially distinct ribbon tilings are possible; these are shown in figure~\ref{fig:2224_ribbons}.  
The example in figure~\ref{subfig:2224_ribbon_a} can be generated by deleting suitable subsets of edges from six fundamental domains, namely those of figures \ref{subfig:2224_c} and \ref{subfig:2224_e}--\ref{subfig:2224_i}.
The one in figure~\ref{subfig:2224_ribbon_b} can be built from each of the nine fundamental domain tilings with multiple distinct edge deletions from some domains giving the same tile-class.  
The branched ribbon tiling in figure~\ref{subfig:2224_ribbon_c} can be found in eight of the fundamental domain tilings: the minimal triangular fundamental domain outlined in figure~\ref{subfig:2224_a} cannot support this branched ribbon. 
Figure \ref{fig:2224_ribbonsex2} illustrates that this fundamental domain supports only one type of ribbon tile, the one with stabiliser \orb{22\infty} shown in figure~\ref{subfig:2224_ribbon_b}.   

Unlike the Euclidean \orb{22\star} example given in the previous section, all three classes of ribbon tilings in figure \ref{fig:2224_ribbons} can be found within a single fundamental domain tiling. 
For example, the fundamental domains in figures \ref{subfig:2224_c} and \ref{subfig:2224_f}  support all three ribbon tilings. 

\begin{figure}[!tbp]
  \begin{subfigure}[t]{0.3\textwidth}
  \centering
    \includegraphics[width=\textwidth]{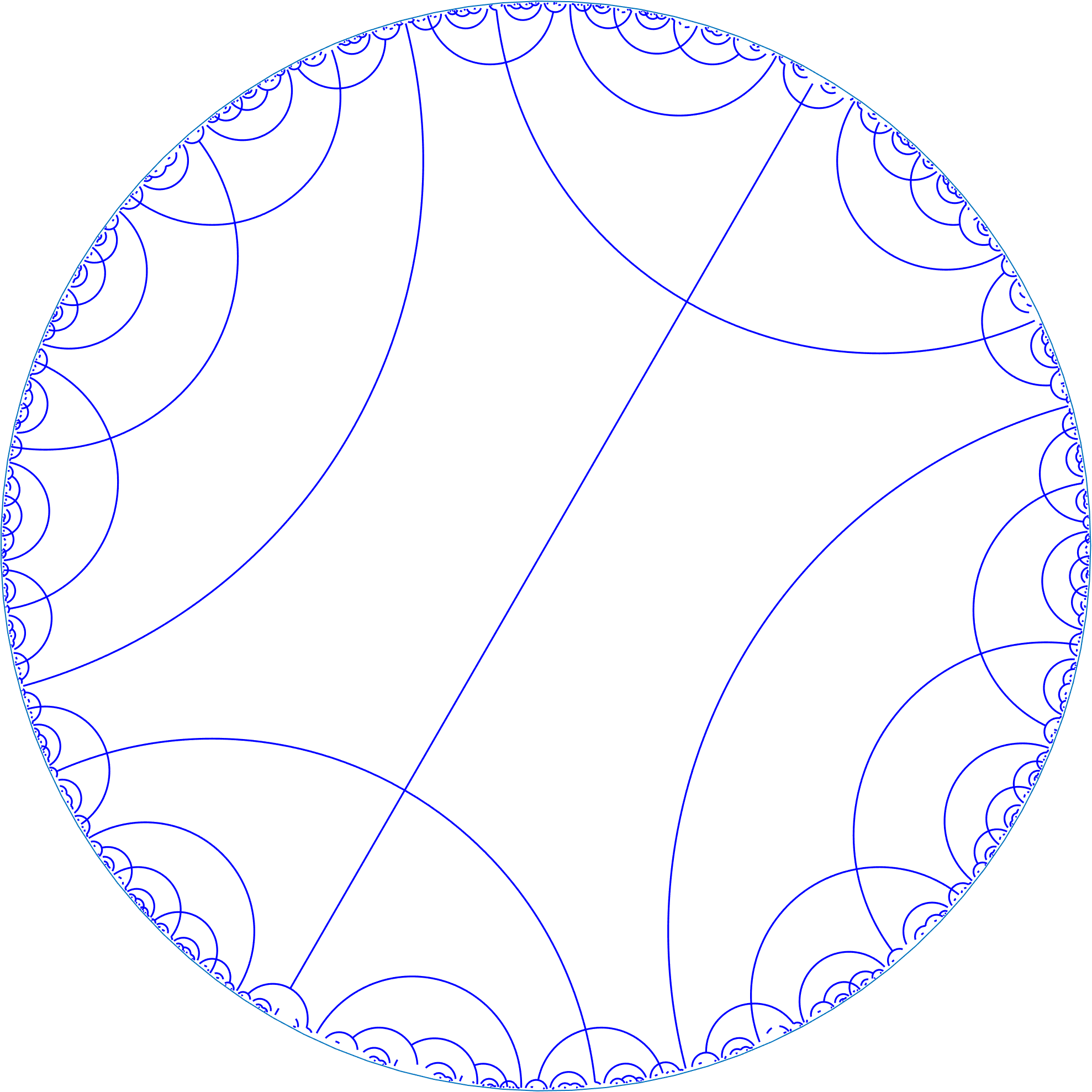}
    \caption{Ribbon tiling with stabiliser group \orbcap{\infty\infty}. }\label{subfig:2224_ribbon_a}
  \end{subfigure}
\hfill
  \begin{subfigure}[t]{0.3\textwidth}
  \centering
    \includegraphics[width=\textwidth]{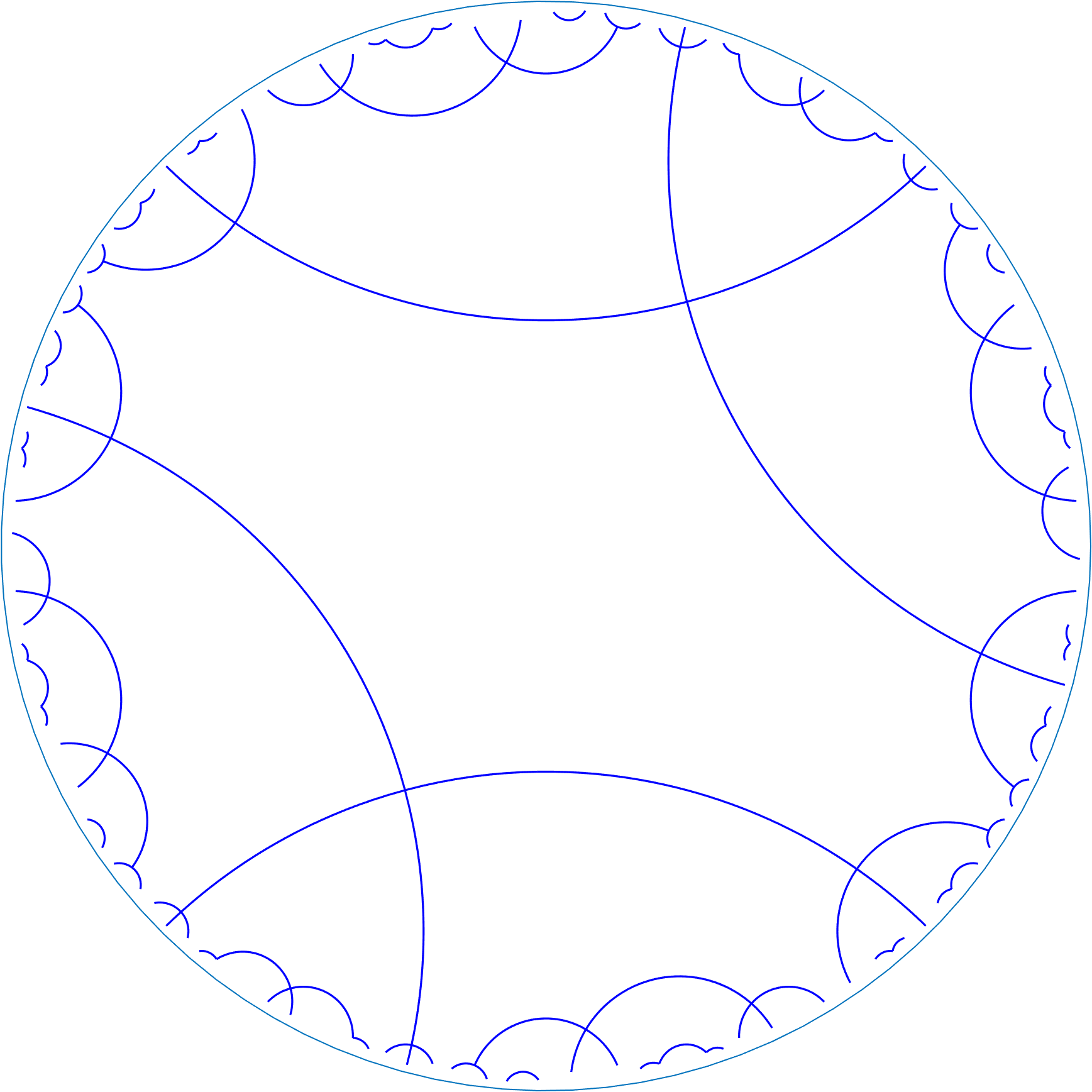}
  \caption{Ribbon tiling with stabiliser group \orbcap{2 2 \infty}.  } \label{subfig:2224_ribbon_b}
  \end{subfigure}
    \hfill
  \begin{subfigure}[t]{0.3\textwidth}
  \centering
    \includegraphics[width=\textwidth]{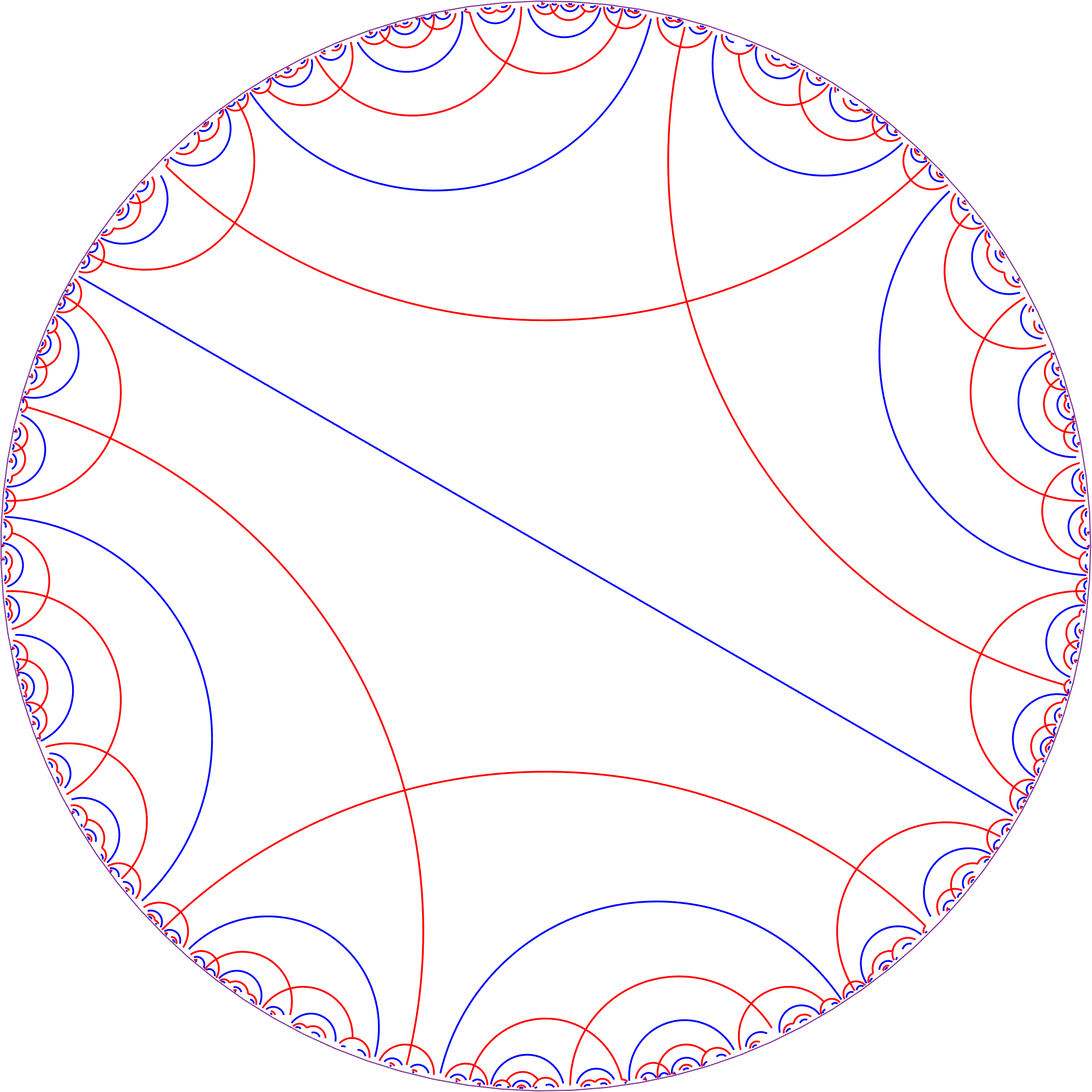}
  \caption{Branched ribbon tiling with stabiliser group \orbcap{2 4 \infty}.  The medial axis is drawn in red to show the branching structure. }\label{subfig:2224_ribbon_c}
  \end{subfigure}
  \caption{The three distinct classes of ribbon tilings with symmetry group \orb{2224}. }\label{fig:2224_ribbons}
\end{figure}

\begin{figure}[!tbp]
  \begin{subfigure}[t]{0.3\textwidth}
  \centering
    \includegraphics[width=\textwidth]{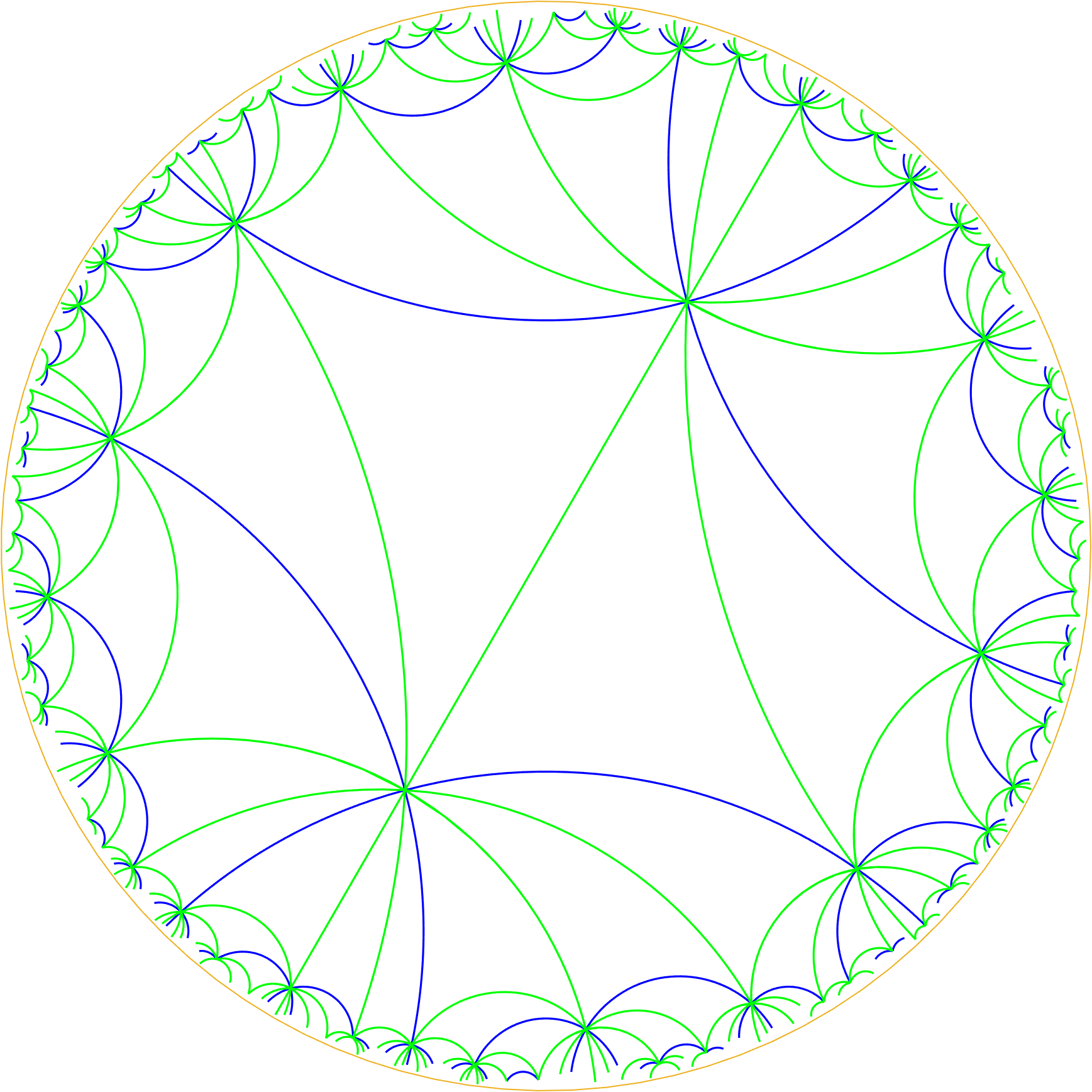}
    \caption{}\label{subfig:2224_simplefd_a}
  \end{subfigure}
\hfill
  \begin{subfigure}[t]{0.3\textwidth}
  \centering
 \includegraphics[width=\textwidth]{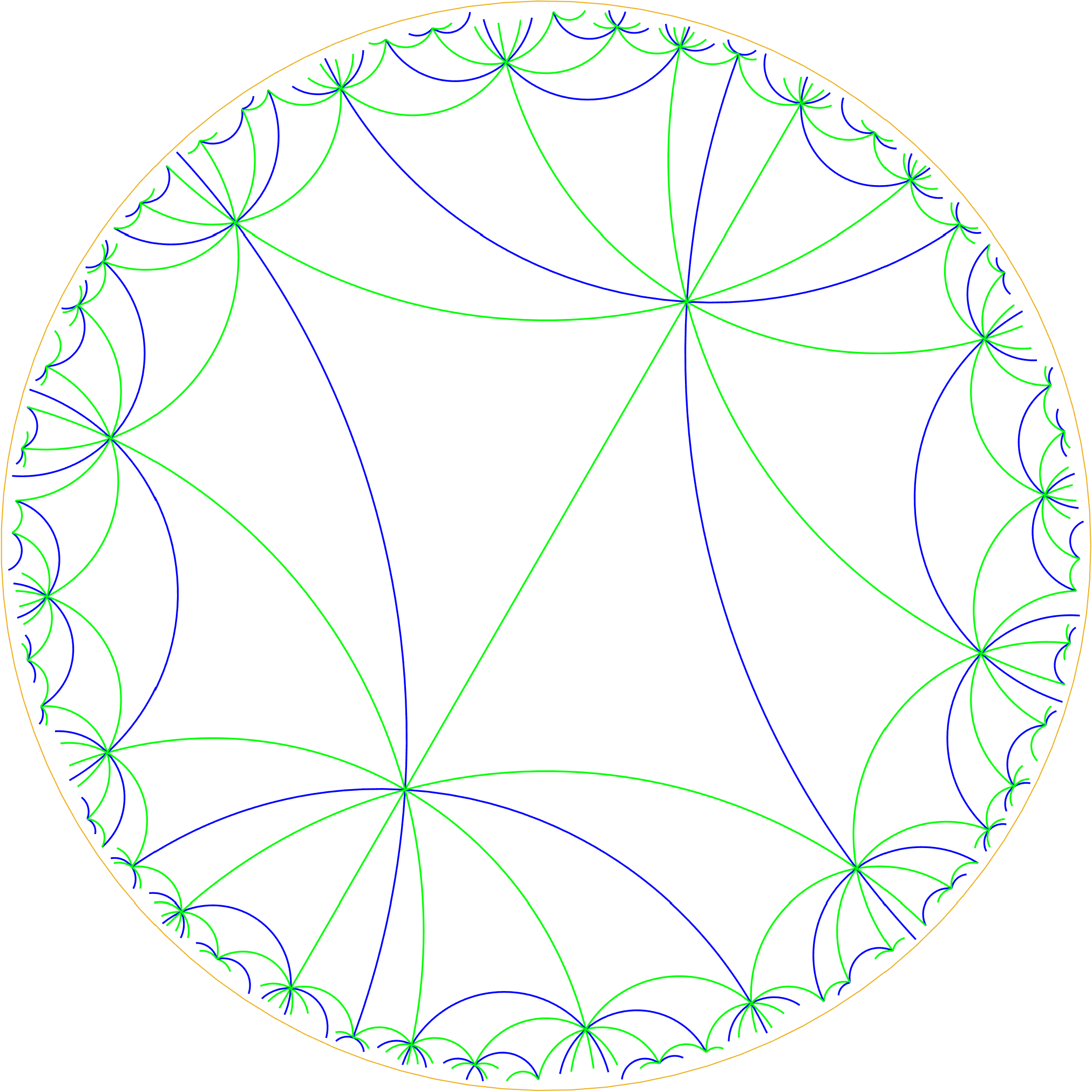}
    \caption{}\label{subfig:2224_simplefd_b}
  \end{subfigure}
  \hfill
  \begin{subfigure}[t]{0.3\textwidth}
  \centering
 \includegraphics[width=\textwidth]{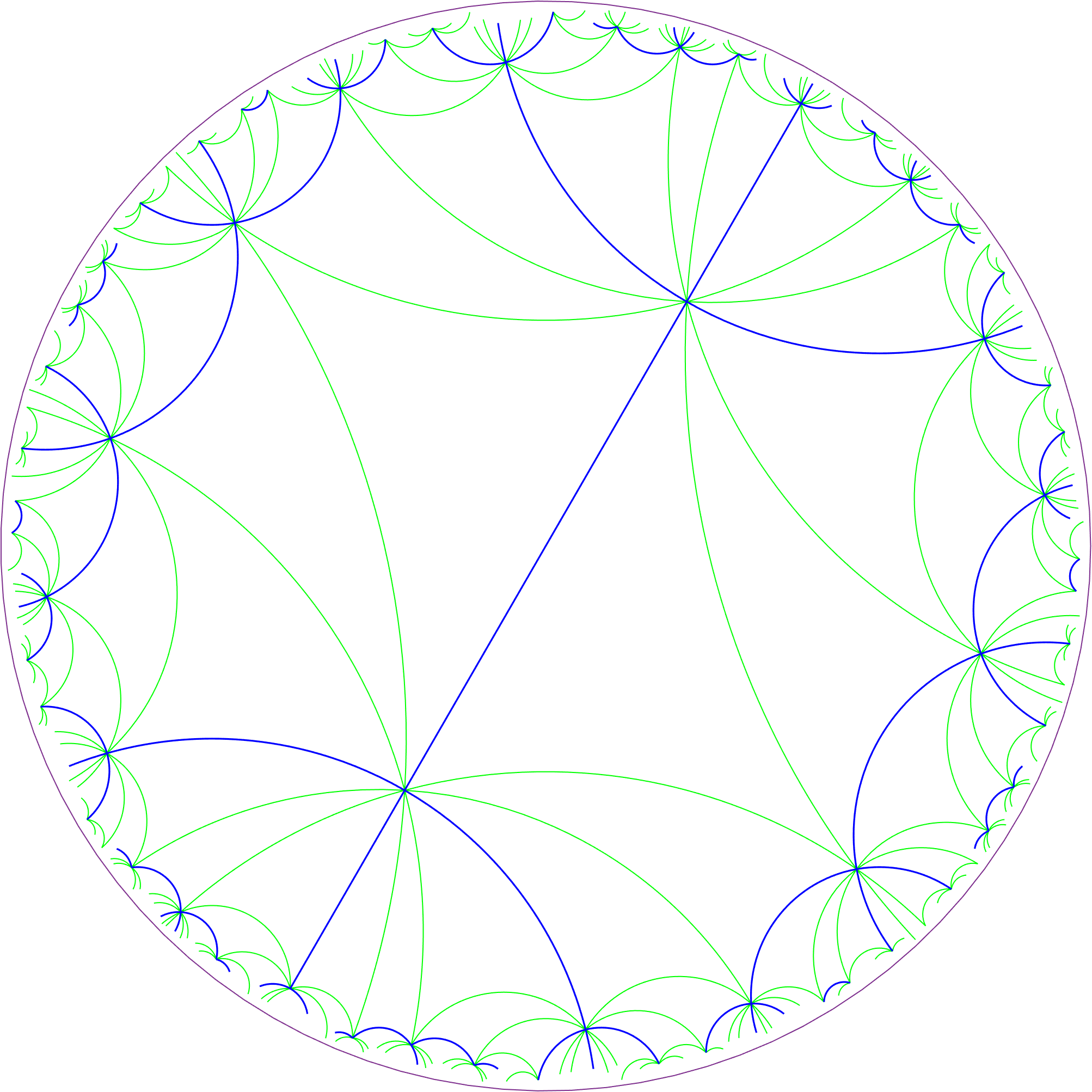}
    \caption{}\label{subfig:2224_simplefd_c}
  \end{subfigure}
  \caption{Deleting any pair of edges from the triangular fundamental domain of figure \ref{subfig:2224_a} leads to a ribbon tiling with stabiliser \orbcap{22\infty}.  Each of these three tilings is equivalent to that in figure \ref{subfig:2224_ribbon_b}}\label{fig:2224_ribbonsex2}
\end{figure}

 Finally, we provide a few examples of tile-$2$-transitive ribbon tilings.  Figure \ref{fig:2224splits} shows two examples of tile split operations that each add a new edge to the fundamental domain of figure \ref{subfig:2224_f}.
The tiling in figure \ref{subfig:2224split3} shows an example of a non-classical tiling that with a finite stabilizer group, $\orb{4\bullet}=C_4$.  Deleting the green edges in this case results in a tile with the topology of an annulus.

\begin{figure}[!tbp]
  \begin{subfigure}[t]{0.3\textwidth}
  \centering
    \includegraphics[width=\textwidth]{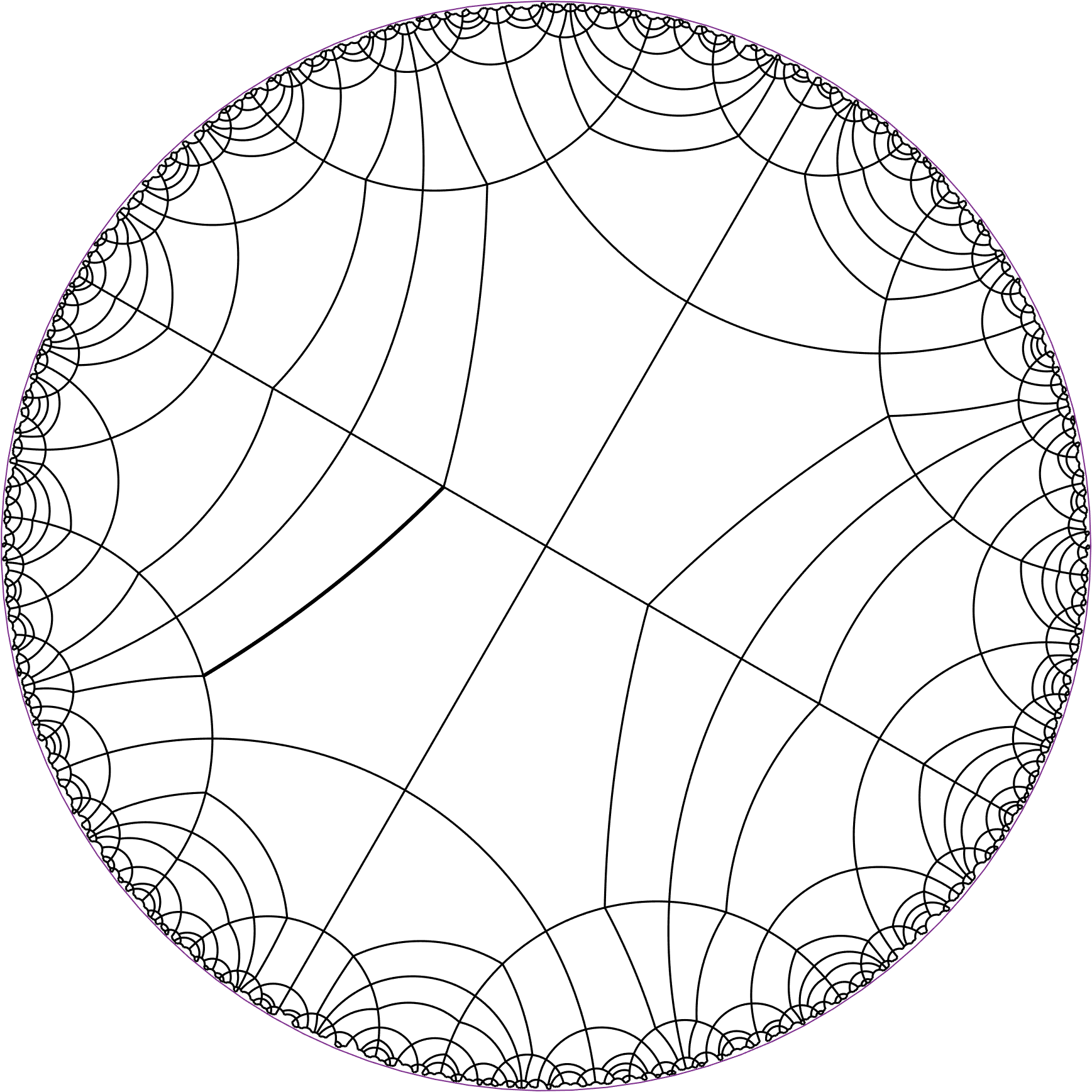}
    \caption{Tile-$2$-transitive tiling with symmetry group \orbcap{2224} obtained by splitting the tiling in figure \ref{subfig:2224_f}.}\label{subfig:2224split1}
  \end{subfigure}
\hfill
  \begin{subfigure}[t]{0.3\textwidth}
  \centering
    \includegraphics[width=\textwidth]{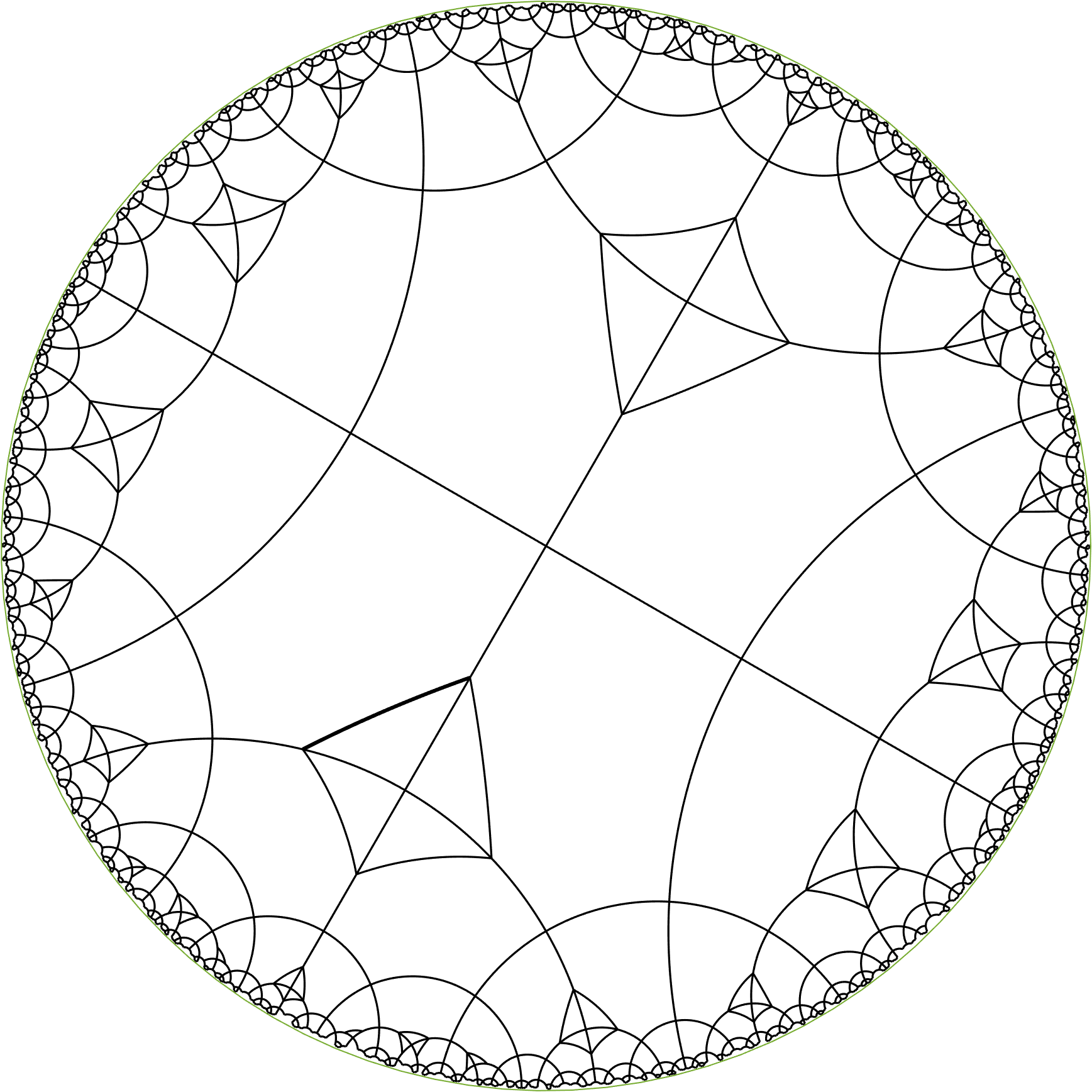}
  \caption{A different way of splitting the fundamental tile in the tiling in figure \ref{subfig:2224_f}.}\label{subfig:2224split2}
  \end{subfigure}
    \hfill
  \begin{subfigure}[t]{0.3\textwidth}
  \centering
    \includegraphics[width=\textwidth]{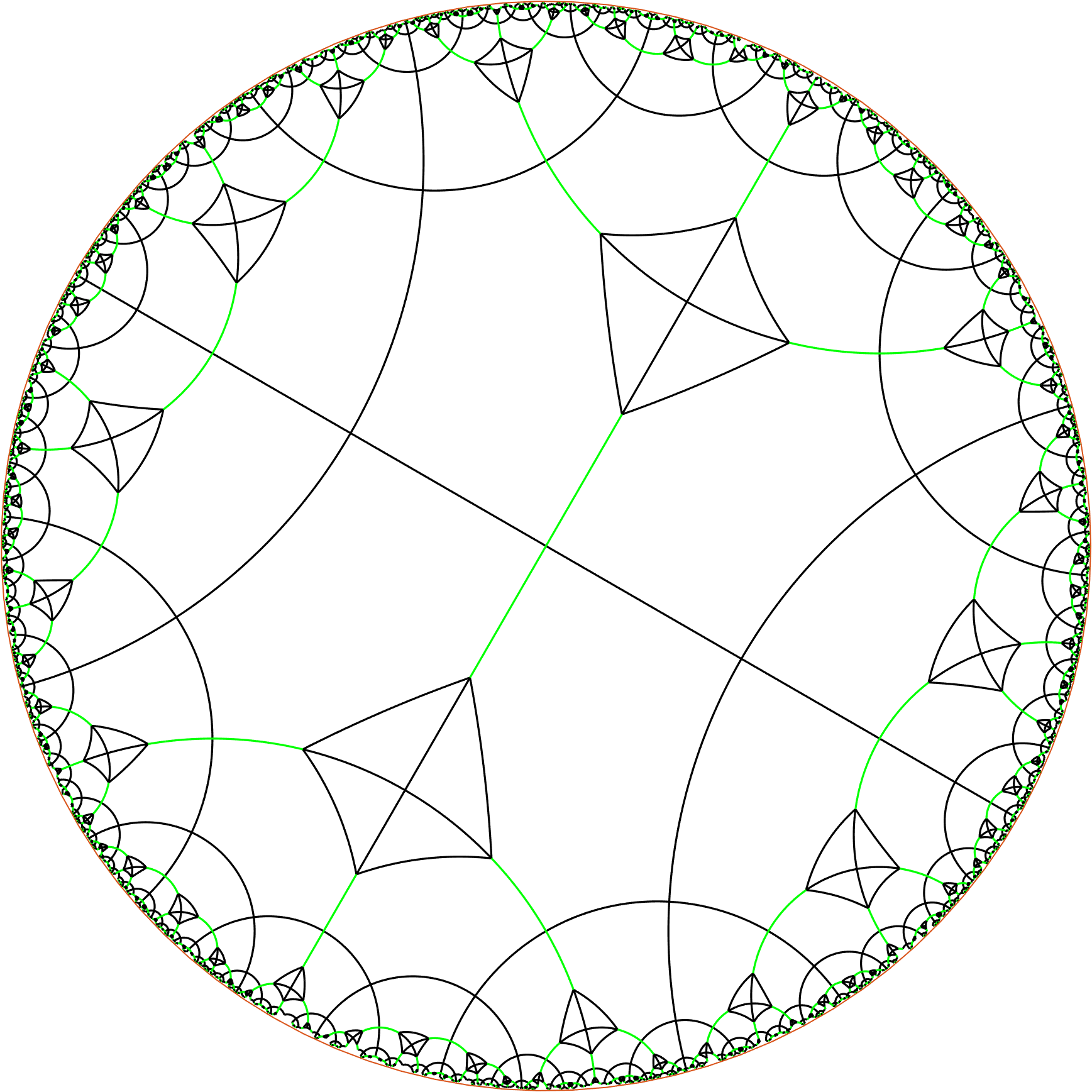}
  \caption{The ribbon tiling resulting from removing the green edges from the tiling in (b), resulting in a ribbon tiling with stabilizer group $\orb{4\bullet}$. }\label{subfig:2224split3}
  \end{subfigure}
  \caption{Tile-$2$-transitive fundamental tilings with symmetry group \orb{2224} and an associated non-fundamental tiling.}\label{fig:2224splits}
\end{figure}

Removing different edges from these split tilings gives the examples in figure \ref{fig:2224tile2trans}.  
The tilings in figures \ref{subfig:2224tile2transa} and \ref{subfig:2224tile2transb} each have one fundamental tile and one non-fundamental tile with infinite stabiliser group, while \ref{subfig:2224tile2transc} has two non-fundamental tiles. 
The ribbon tiles in \ref{subfig:2224tile2transb} and \ref{subfig:2224tile2transc} have stabiliser group \orb{24\infty} and are non-simply connected.   In fact, the fundamental group of this tile is not even finitely generated. 

\begin{figure}[!tbp]
  \begin{subfigure}[t]{0.3\textwidth}
  \centering
    \includegraphics[width=\textwidth]{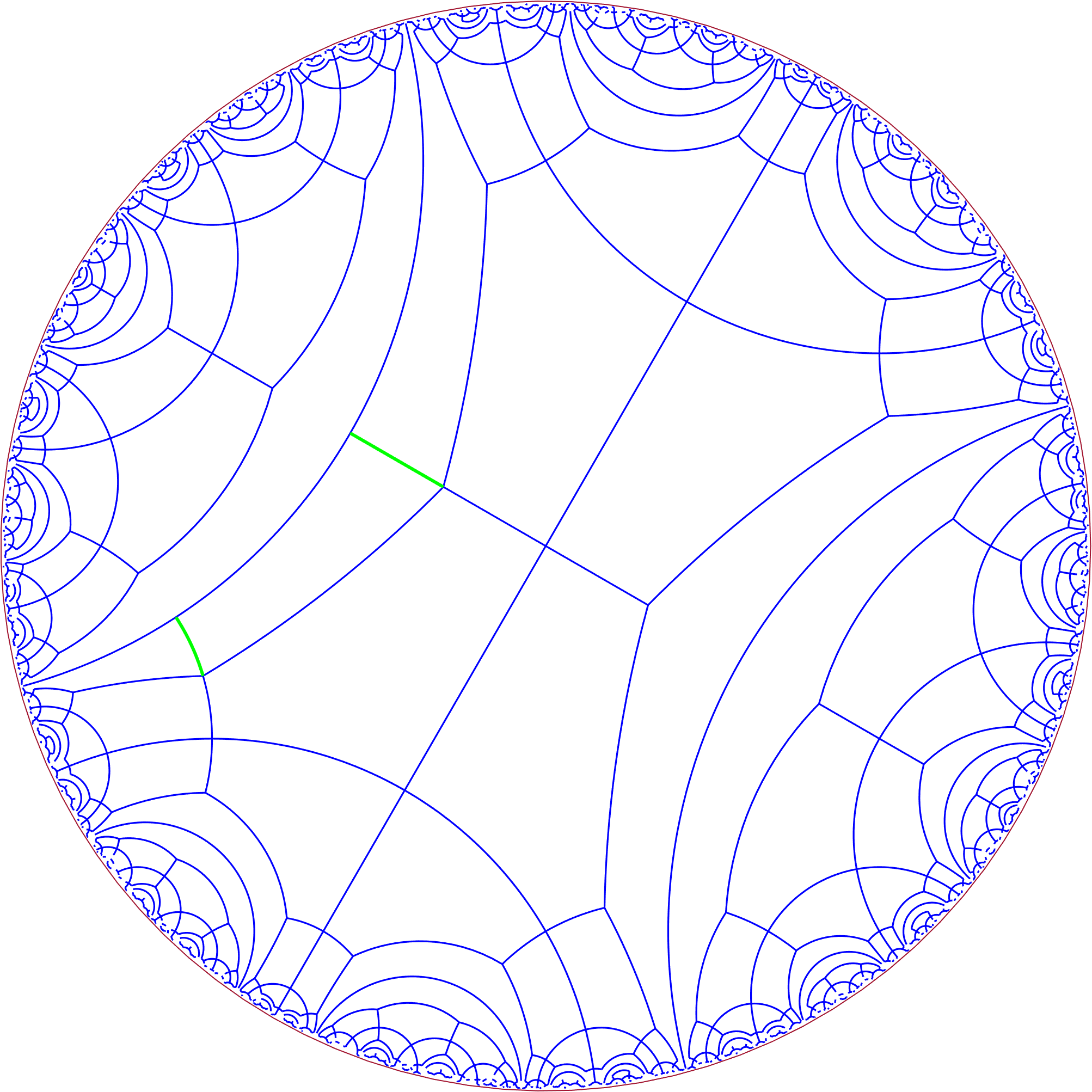}
    \caption{A tile-$2$-transitive ribbon tiling with a non-fundamental tile obtained by deleting the green edges from the split tiling in figure \ref{subfig:2224split1}. The ribbon has stabilizer group \orbcap{\infty\infty}.}\label{subfig:2224tile2transa}
  \end{subfigure}
  \hfill
  \begin{subfigure}[t]{0.3\textwidth}
  \centering
    \includegraphics[width=\textwidth]{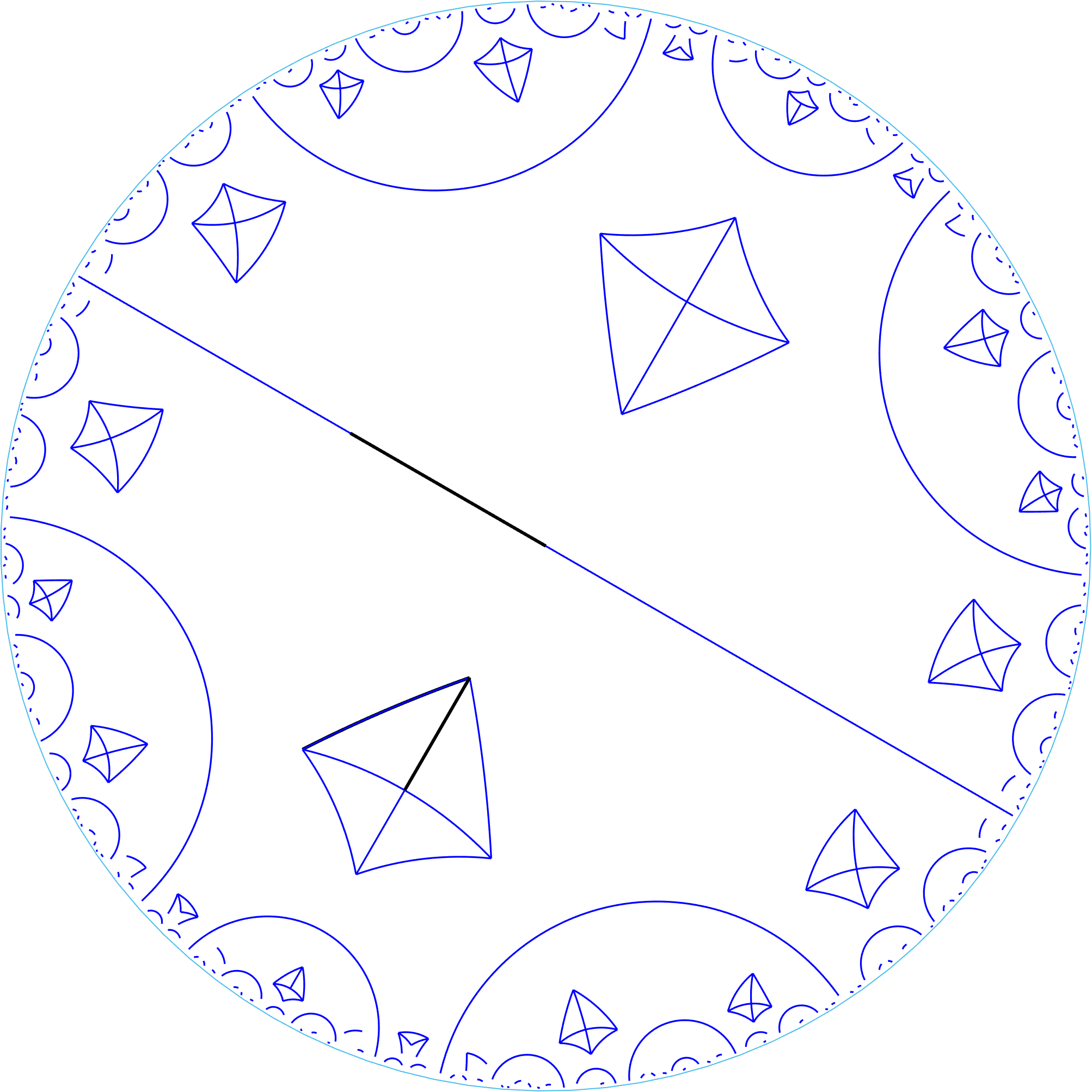}
  \caption{A tile-$2$-transitive ribbon tiling with a non-fundamental tile obtained by deleting an edge from the split tiling in figure \ref{subfig:2224split3}. The ribbon has stabilizer group \orbcap{24\infty}.}\label{subfig:2224tile2transb}
  \end{subfigure}
    \hfill
  \begin{subfigure}[t]{0.3\textwidth}
  \centering
    \includegraphics[width=\textwidth]{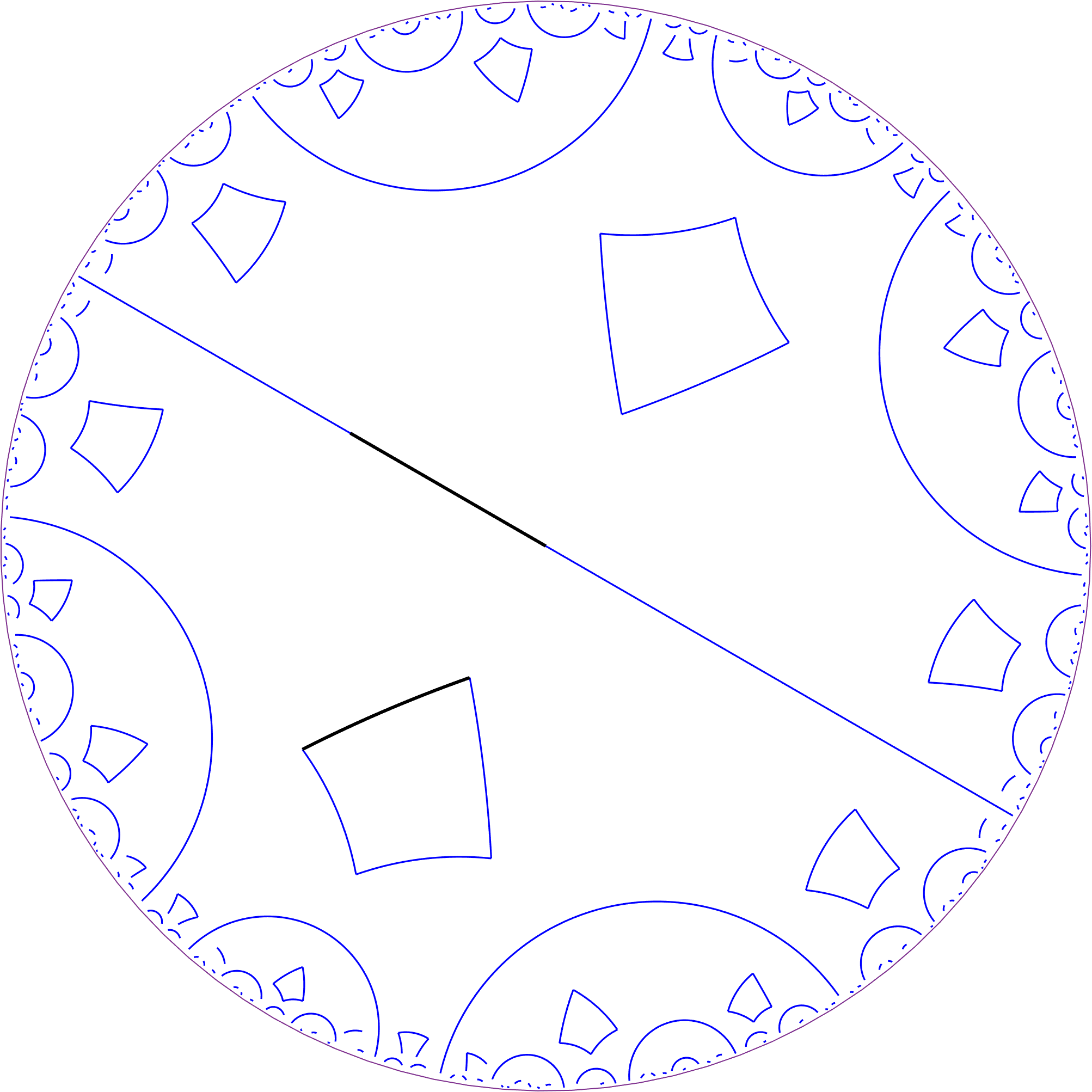}
  \caption{Gluing copies of the remaining fundamental tile in (b) produces a second non-fundamental tile with stabilizer group $\orbcap{4\bullet}$ }\label{subfig:2224tile2transc}
  \end{subfigure}
  \caption{Tile-2-transitive ribbon tilings with symmetry group 2224.}\label{fig:2224tile2trans}
\end{figure}

\section{Summary and Outlook} 

In this paper, we showed how to enumerate and classify periodic, locally finite tilings of the Euclidean or hyperbolic plane $\mathcal{X}$, possibly with unbounded tiles with nontrivial topology. 
The enumeration is based on gluing together fundamental domain tiles to create non-fundamental ones, possibly with stabiliser groups of infinite order.  
The method adapts well to the sequential enumerative aspects of the classical theory estabilished in~\cite{Huson1993} because the new glue operations naturally extend to split fundamental domain tilings.  

The classification builds on the classical D-symbol by introducing coloured edges that cut a ribbon or annulus tile into symmetry-related fundamental disks. 
The coloured D-symbol encoding a ribbon tiling is isomorphic to another coloured D-symbol if and only if the tilings may be deformed into one another while preserving their abstract symmetry. 
When constructing the coloured D-symbol, it is necessary to construct the possible trees on $n$ vertices. 
For distinguishable vertices, it is well-known that there are $n^{n-2}$ spanning trees. 
As illustrated in section \ref{sec:ex}, it is often possible to discard some trees in the construction of the D-symbol. 
However, to prevent a combinatorial explosion in the classification algorithm, future work should focus on \emph{a priori} estimates of the D-symbol in this enumeration and explore more efficient approaches to constructing the coloured D-symbol of a given ribbon tiling. 

Note also that the theory developed here extends to the situation of hyperbolic symmetry groups $\Gamma$, where the quotient space  
$\mathcal{X}/\Gamma$ has finite area but is not compact. 
The situation is exactly the same as before and it makes sense to interpret the punctures that appear in the orbifold as gyration points of infinite order (i.e. parabolic isometries of $\mathcal{X}$). 
For example, it would still be the case in this more general setting that deleting two non-neighbouring edges of a fundamental tile-$1$-transitive tiling would lead to a tile with a translation in its stabilizer subgroup, and consequently a tile with infinite area. 

Finally, we observe that the classification algorithm developed here could also be used to study graph embeddings in compact surfaces when the embedding is not a combinatorial map. 

And returning to the initial inspiration of this work --- the question of how to enumerate stripe patterns on the gyroid --- this can now be achieved by finding (branched) ribbon tilings in the symmetry groups compatible with the covering map that wraps the hyperbolic plane onto this periodic surface.  

\bigskip
\noindent
\textbf{Acknowledgements:} 
The authors would like to thank Myfanwy Evans from the Technical University in Berlin, and Stephen Hyde and Stuart Ramsden from the Australian National University for fruitful discussions and enlivening talks about the problem of classifying ribbon tilings over many years. 

\bibliographystyle{spmpsci}  
\bibliography{bibliography}  

\end{document}